\documentclass{amsart}
\usepackage{amsfonts}
\usepackage[dvips]{graphicx}
\usepackage[b5paper,twoside,top=20mm,right=18mm,left=22mm,bottom=15mm,bindingoffset=0mm,nomarginpar]{geometry}
\usepackage[T1]{fontenc}

\newtheorem{theorem}{Theorem}[section]

\newtheorem{lemma}{Lemma}[section]
\newtheorem{remark}{Remark}[section]
\newtheorem{corollary}{Corollary}[section]

\numberwithin{equation}{section}
\email{xhevat.krasniqi@uni-pr.edu}
\email{korus.peter@szte.hu}
\email{B.Szal@wmie.uz.zgora.pl}
\subjclass[2020]{40G05; 41A25; 42A10.}
\keywords{Degree of approximation, H\"older metric, modulus of continuity, delayed arithmetic
mean, Minkowski inequality, Fourier series.}
\dedicatory{}

\begin{document}
\title[%
\uppercase{Approximation by double second type delayed arithmetic
mean}]{%
\uppercase{Approximation by double second type delayed arithmetic
mean of periodic functions in $H_{p}^{(\omega, \omega)}$ space}}
\author[Xh. Z. Krasniqi]{Xh. Z. Krasniqi}
\address{Faculty of Education \\
\indent University of Prishtina "Hasan Prishtina" \\
\indent Avenue "Mother Theresa" 5, 10000 Prishtina \\
\indent Kosovo}
\author[P. K\'orus]{P. K\'orus}
\address{Institute of Applied Pedagogy \\
\indent Juh\'asz Gyula Faculty of Education \\
\indent University of Szeged \\
\indent Hattyas utca 10, H-6725 Szeged \\
\indent Hungary}
\author[B. Szal]{B. Szal}
\address{Faculty of Mathematics, Computer Science and Econometrics \\
\indent University of Zielona G\'{o}ra \\
\indent ul. Szafrana 4a, 65-516 Zielona G\'{o}ra \\
\indent Poland}
\date{}

\begin{abstract}
In this paper, we give a degree of approximation of a function in the space $H_{p}^{(\omega, \omega)}$ by using the second type double delayed arithmetic means of its Fourier series. Such degree of approximation is expressed  via two functions of moduli of continuity type. To obtain one more general result, we used the even-type double delayed arithmetic means of Fourier series as well.
\end{abstract}

\maketitle



\section{Concise historical background and motivation}

On one hand, the approximation of $2\pi$-periodic and integrable functions by their Fourier series in the H\"{o}lder metric has been studied widely and consistently in many papers. Das at al. studied the degree of approximation of functions by matrix means of their Fourier series in the generalized H\"{o}lder metric \cite{DGR}, generalizing many previous known results. Again, Das at al. \cite{DNR} studied the rate of convergence problem of Fourier series in a new Banach space of functions conceived as a generalization of the spaces introduced by Pr\"{o}ssdorf \cite{SP} and Leindler \cite{L1}. Afterwards, Nayak at al. \cite{NDR} studied the rate of convergence problem of the Fourier series by delayed arithmetic mean in the generalized H\"{o}lder metric space which was earlier introduced in \cite{DNR} and obtaining a sharper estimate of Jackson's order which is the main objective of their result. Moreover, De\v{g}er \cite{D} determined the degree of approximation of functions by matrix means of their Fourier series in the same space of functions introduced in \cite{DNR}. In particular, he extended some results of Leindler \cite{L} and some other results by weakening the monotonicity conditions in the results obtained by Singh and Sonker \cite{SS} for some classes of numerical sequences introduced by Mohapatra and Szal \cite{MS}.  Leindler's results obtained in \cite{L}, are generalized in \cite{XhK} by first author of the present paper, for functions from a Banach space, mainly using the generalized N\"{o}rlund and Riesz means. Very recently,  Kim \cite{KIM1} presented a generalized result of a particular case of a result obtained previously in \cite{NDR}. In Kim's result is treated the degree of approximation of functions in the same generalized H\"{o}lder metric, but using the so-called even-type delayed arithmetic mean of Fourier series.   

On the other hand, results on approximation of bivariate integrable functions and $2\pi$-periodic in each variable by their double Fourier series in the H\"{o}lder metric, the interested reader can find in \cite{UD2}, \cite{XhK2} and \cite{NH}. In all results reported in these papers, we came across to the degree of approximation of functions by various means of their double Fourier series and in which the quantity of the form $\mathcal{O}{(\log n)}$ appears. Involving such quantity produces a degree of approximation which is not of Jackson's order. This weakness motivated us to consider some means of double Fourier series which will overshoots it. Whence, we are going to investigate the degree of approximation of bivariate integrable functions and $2\pi$-periodic in each variable by their double Fourier series in the generalized H\"{o}lder metric, by motivation of removing the quantities of the form $\mathcal{O}{(\log n)}$ and obtaining the degree of approximation of Jackson's order, which is the aim of the paper.

Closing this section, for comparing of two quantities $u$ and $v>0$, throughout this paper we write $u=\mathcal{O}(v)$, whenever there exists a positive constant $c$ such that $u\leq cv$.

\section{Introduction and preliminaries}

By $L_p(T^2)$, $p\geq 1$, we denote the space of all functions $f(x,y)$
integrable with $p$-power on $T^2:=(0,2\pi)\times (0,2\pi)$, and with norm 
\begin{equation*}
\|f\|_p:=\left(\frac{1}{(2\pi)^2}\int_{0}^{2\pi}\int_{0}^{2\pi}|f(x,y)|^p
dxdy\right)^{1/p}.
\end{equation*}
If $f\in L_p(T^2)$, $p\geq 1$, then $\omega _i$, $(i=1,2)$ are considered as
moduli of continuity if $\omega _i$ are two positive non-decreasing
continuous functions on $[0,2\pi]$ with properties
\begin{enumerate}
\item[(i)] $\omega _i(0)=0$,

\item[(ii)] $\omega _i(\delta_1+\delta_2)\leq \omega _i(\delta_1)+\omega
_i(\delta_2)$,

\item[(iii)] $\omega _{i}(\lambda \delta )\leq (\lambda +1)\omega
_{i}(\delta )$, $\lambda \geq 0$.
\end{enumerate}
We define the space $H_{p}^{(\omega_1, \omega_2)}$ by 
\begin{equation*}
H_{p}^{(\omega_1, \omega_2 )}:=\left\{f\in L^{p}(T^2), p\geq 1:
A(f;\omega_1, \omega_2 )<\infty \right\},
\end{equation*}
where 
\begin{equation*}
A(f;\omega_1, \omega_2 ):=\sup_{t_1\neq 0,\,\,t_2\neq 0 }\frac{\|f(x +t_1,y
+t_2)-f(x,y)\|_{p}}{\omega_1 (|t_1|)+\omega_2 (|t_2|)}
\end{equation*}
and the norm in the space $H_{p}^{(\omega_1, \omega_2)}$ is defined by 
\begin{equation*}
\|f\|_{p}^{(\omega_1, \omega_2)}:=\|f\|_{p}+A(f;\omega_1, \omega_2 ).
\end{equation*}
If $\omega_1$, $\omega_2$, $v_1$ and $v_2$ are moduli of continuity so that
the two-variable function $\frac{\omega_1 (t_1)+\omega_2 (t_2)}{%
v_1(t_1)+v_2(t_2)}$ has a maximum $M$ on $T^2$, then it is easy to see that 
\begin{equation*}
\|f\|_{p}^{(v_1,v_2)}\leq \max \left(1,M\right)\|f\|_{p}^{(\omega_1,
\omega_2)},
\end{equation*}
which shows that in this case, for the given spaces $H_{p}^{(\omega_1,
\omega_2)}$ and $H_{p}^{(v_1, v_2)}$ we have 
\begin{equation*}
H_{p}^{(\omega_1, \omega_2)}\subseteq H_{p}^{(v_1, v_2)}\subseteq L_p \quad
(p\geq 1).
\end{equation*}

We write 
\begin{equation*}
\Omega_p(\delta_1,\delta_2;f):=\sup_{0\leq h_1\leq \delta_1; 0\leq h_2\leq
\delta_2}\|f(x+h_1,y+h_2)-f(x,y)\|_{p}
\end{equation*}
for the integral modulus of continuity of $f(x,y)$, and whenever 
\begin{equation*}
\|f(x +t_1,y +t_2)-f(x,y)\|_{p}=\mathcal{O}\left(\omega_1 (|t_1|)+\omega_2
(|t_2|)\right)
\end{equation*}
we write $f\in \text{Lip}(\omega_1,\omega_2,p)$, that is 
\begin{equation*}
\text{Lip}(\omega_1,\omega_2,p)=\left\{f\in L_p(T^2): \|f(x +t_1,y
+t_2)-f(x,y)\|_{p}=\mathcal{O}\left(\omega_1 (|t_1|)+\omega_2
(|t_2|)\right)\right\}.
\end{equation*}
Clearly, for $\omega_1(t_1)=\mathcal{O}\left(t_1^{\alpha}\right)$ and $%
\omega_2(t_2)=\mathcal{O}\left(t_2^{\beta}\right)$, $0<\alpha \leq 1$, $%
0<\beta \leq 1$, the class $\text{Lip}(\omega_1,\omega_2,p)$ reduces to the
class $\text{Lip}(\alpha,\beta,p)$, that is 
\begin{equation*}
\text{Lip}(\alpha,\beta, p)=\left\{f\in L^p(T^2): \|f(x +t_1,y
+t_2)-f(x,y)\|_{p}=\mathcal{O}\left(t_1^{\alpha}\right)+\mathcal{O}%
\left(t_2^{\beta}\right)\right\}.
\end{equation*}
Then for $1 \geq \alpha \geq \gamma \geq 0$ and $1 \geq \beta \geq \delta
\geq 0$, by noting $\frac{t_1^{\alpha}+t_2^{\beta}}{t_1^{\gamma}+t_2^{\delta}%
}$ is bounded on $T^2$, we have 
\begin{equation*}
\text{Lip}(\alpha,\beta,p)\subseteq \text{Lip}(\gamma,\delta,p) \subseteq
L_p \quad (p\geq 1).
\end{equation*}
If $f(x,y)$ is a continuous function periodic in both variables with period $%
2\pi$ and $p=\infty$, then the class $\text{Lip}(\alpha,\beta,p)$ reduces to
the H\"older class $\text{H}_{(\alpha,\beta)}$ (also called Lipschitz
class), that is 
\begin{equation*}
\text{H}_{(\alpha,\beta)}=\left\{f: |f(x +t_1,y +t_2)-f(x,y)|=\mathcal{O}%
\left(t_1^{\alpha}\right)+\mathcal{O}\left(t_2^{\beta}\right)\right\}.
\end{equation*}
It is verified that $\text{H}_{(\alpha,\beta)}$ is a Banach space (see \cite%
{UD2}) with the norm $\|f \|_{\alpha,\beta}$ defined by 
\begin{equation*}
\|f \|_{\alpha,\beta}=\|f\|_{C}+\sup_{t_1\neq 0,\,\,t_2\neq 0 }\frac{|f(x
+t_1,y +t_2)-f(x,y)|}{|t_1|^{\alpha}+|t_2|^{\beta}},
\end{equation*}
where 
\begin{equation*}
\|f\|_{C}=\sup_{(x,y)\in T^2}|f(x,y)|.
\end{equation*}

Let $f(x,y)\in L^p(T^2)$ be a $2\pi$-periodic function with respect to each
variable, with its Fourier series 
\begin{align*}
& f(x,y)\sim
\sum_{k=0}^{\infty}\sum_{\ell=0}^{\infty}\lambda_{k,\ell}(a_{k,\ell} \cos kx
\cos \ell y+b_{k,\ell} \sin kx \cos \ell y \\
& \hspace{3cm}+c_{k,\ell} \cos kx \sin \ell y+d_{k,\ell} \sin kx \sin \ell
y)
\end{align*}
at the point $(x,y)$, where 
\begin{align*}
\lambda _{k,\ell} & = \left\{ 
\begin{array}{rcl}
\frac{1}{4}, & \mbox{if} & k=\ell=0, \\ 
\frac{1}{2}, & \mbox{if} & k=0, \,\ell>0;\,\,\,\mbox{or} \,\,\, k>0,
\,\ell=0, \\ 
1, & \mbox{if} & k, \,\ell>0;%
\end{array}%
\right. \\
a_{k,\ell} & = \frac{1}{\pi^{2}}\int_{-\pi}^{\pi}\int_{-\pi}^{\pi}f(u, v)
\cos ku \cos \ell v \,du dv, \\
b_{k,\ell} & = \frac{1}{\pi^{2}}\int_{-\pi}^{\pi}\int_{-\pi}^{\pi}f(u, v)
\sin ku \cos \ell v \,du dv, \\
c_{k,\ell} & = \frac{1}{\pi^{2}}\int_{-\pi}^{\pi}\int_{-\pi}^{\pi}f(u, v)
\cos ku \sin \ell v \,du dv, \\
d_{k,\ell} & = \frac{1}{\pi^{2}}\int_{-\pi}^{\pi}\int_{-\pi}^{\pi}f(u,
v)\sin ku \sin \ell v\, du dv,\quad k,\,\ell\in \mathbb{N}\cup\{0\},
\end{align*}
and whose partial sums are 
\begin{align*}
& S_{n,m}(f;x,y)= \sum_{k=0}^{n}\sum_{\ell=0}^{m}\lambda_{k,\ell}(a_{k,\ell}
\cos kx \cos \ell y+b_{k,\ell} \sin kx \cos \ell y \\
& \hspace{3.8cm}+c_{k,\ell} \cos kx \sin \ell y+d_{k,\ell} \sin kx \sin \ell
y),\quad n,m\geq 0 .
\end{align*}

To reveal our intention, we recall some other notations and notions.

Let $\sum_{k=0}^{\infty}\sum_{\ell=0}^{\infty}u_{k,\ell}$ be an infinite
double series with its sequence of arithmetic mean $\{\sigma_{m,n}\}$, where 
\begin{equation*}
\sigma_{m,n}=\frac{1}{(m+1)(n+1)}\sum_{k=0}^{m}\sum_{\ell=0}^{n}U_{k,\ell}
\end{equation*}
and $U_{k,\ell}:=\sum_{i=0}^{k}\sum_{j=0}^{\ell}u_{i,j}.$

We define the \textit{Double Delayed Arithmetic Mean} $\sigma_{m,k;n,\ell}$
by (see \cite{FW}): 
\begin{align}  \label{eq1}
\begin{split}
\sigma_{m,k;n,\ell}:=&\left(1+\frac{m}{k}\right)\left(1+\frac{n}{\ell}%
\right)\sigma_{m+k-1,n+\ell-1}-\left(1+\frac{m}{k}\right)\frac{n}{\ell}%
\sigma_{m+k-1,n-1} \\
& -\frac{m}{k}\left(1+\frac{n}{\ell}\right)\sigma_{m-1,n+\ell-1}+\frac{mn}{%
k\ell}\sigma_{m-1,n-1},
\end{split}
\end{align}
where $k$ and $\ell$ are positive integers.

If $k$ tends to $\infty$ with $m$ in such a way that $\frac{m}{k}$ is
bounded, and $\ell$ tends to $\infty$ with $n$ in such a way that $\frac{n}{%
\ell}$ is also bounded, then $\sigma_{m,k;n,\ell}$ defines a method of
summability which is at least as strong as the well-known $(C,1,1)$
summablity. This means that if $\sigma_{m,n}\to \mu$, then $%
\sigma_{m,k;n,\ell}\to \mu$ as well. This important fact follows from (\ref%
{eq1}) if we set $\sigma_{m,n}=\mu+\xi_{m,n}$, where $\xi_{m,n}\to 0$ as $%
m,n\to \infty$. Introducing this mean we expect to be useful in
applications, particularly in approximating of $2\pi$ periodic functions in
two variables.

We note that for $k=\ell=1$ we obtain $\sigma_{m,1;n,1}=U_{m,n}$, while for $%
m=n=0$ we get $\sigma_{0,k;0,\ell}=\sigma_{k-1,\ell-1}$. Moreover, for $k=m$
and $\ell=n$, we get 
\begin{align}  \label{eq2}
\sigma_{m,m;n,n}=4\sigma_{2m-1,2n-1}-2\sigma_{2m-1,n-1}-2\sigma_{m-1,2n-1}+%
\sigma_{m-1,n-1}
\end{align}
that we name as the \textit{first type} Double Delayed Arithmetic Mean.

However, in this research paper we will take $k=2m$ and $\ell=2n$ in the
Double Delayed Arithmetic Mean $\sigma_{m,k;n,\ell}$ to obtain 
\begin{align}  \label{eq3}
\sigma_{m,2m;n,2n}=\frac{1}{4}\left(9\sigma_{3m-1,3n-1}-3\sigma_{3m-1,n-1}-3%
\sigma_{m-1,3n-1}+\sigma_{m-1,n-1}\right).
\end{align}
We name these particular sums as the \textit{second type} Double Delayed
Arithmetic Mean.

By $\sigma _{m,n}(f;x,y)$ and $\sigma _{m,2m;n,2n}(f;x,y)$ we denote the
arithmetic mean and the second type Double Delayed Arithmetic Mean for $%
S_{k,\ell }(f;x,y)$, respectively. It is well-known (see e.g. \cite[page 4]%
{KIM}) that the double Fej\'{e}r kernel is 
\begin{equation*}
F_{m,n}(t_{1},t_{2}):=\frac{4}{(m+1)(n+1)}\left( \frac{\sin \frac{(m+1)t_{1}%
}{2}\sin \frac{(n+1)t_{2}}{2}}{4\sin \frac{t_{1}}{2}\sin \frac{t_{2}}{2}}%
\right) ^{2},
\end{equation*}%
which we will use it in its equivalent form 
\begin{equation}\label{eq4}
F_{m,n}(t_{1},t_{2}):=\frac{1}{(m+1)(n+1)}\frac{(1-\cos (m+1)t_{1})(1-\cos
(n+1)t_{2})}{\left( 4\sin \frac{t_{1}}{2}\sin \frac{t_{2}}{2}\right) ^{2}},
\end{equation}%
while the arithmetic mean $\sigma _{m,n}(f;x,y)$ is 
\begin{equation}\label{eq5}
\sigma _{m,n}(f;x,y)=\frac{1}{\pi ^{2}}\int_{0}^{\pi }\int_{0}^{\pi
}h_{x,y}(t_{1},t_{2})F_{m,n}(t_{1},t_{2})\,dt_{1}dt_{2},
\end{equation}%
where 
\begin{equation*}
h_{x,y}(t_{1},t_{2}):=f(x+t_{1},y+t_{2})+f(x-t_{1},y+t_{2})+f(x+t_{1},y-t_{2})+f(x-t_{1},y-t_{2}).
\end{equation*}%
Furthermore, using (\ref{eq4}) successively, we have 
\begin{align*}
F_{3m-1,3n-1}(t_{1},t_{2})& =\frac{(1-\cos (3mt_{1}))(1-\cos (3nt_{2}))}{%
9mn\left( 4\sin \frac{t_{1}}{2}\sin \frac{t_{2}}{2}\right) ^{2}}, \\
F_{3m-1,n-1}(t_{1},t_{2})& =\frac{(1-\cos (3mt_{1}))(1-\cos (nt_{2}))}{%
3mn\left( 4\sin \frac{t_{1}}{2}\sin \frac{t_{2}}{2}\right) ^{2}}, \\
F_{m-1,3n-1}(t_{1},t_{2})& =\frac{(1-\cos (mt_{1}))(1-\cos (3nt_{2}))}{%
3mn\left( 4\sin \frac{t_{1}}{2}\sin \frac{t_{2}}{2}\right) ^{2}}, \\
F_{m-1,n-1}(t_{1},t_{2})& =\frac{(1-\cos (mt_{1}))(1-\cos (nt_{2}))}{%
mn\left( 4\sin \frac{t_{1}}{2}\sin \frac{t_{2}}{2}\right) ^{2}}.
\end{align*}%
Thus, we can write 
\begin{equation}
\begin{split}
F_{m,2m;n,2n}(t_{1},t_{2}):=& \frac{1}{4}\left(
9F_{3m-1,3n-1}(t_{1},t_{2})-3F_{3m-1,n-1}(t_{1},t_{2})\right. \\
& \quad \left. -3F_{m-1,3n-1}(t_{1},t_{2})+F_{m-1,n-1}(t_{1},t_{2})\right) \\
=& \frac{1}{4mn\left( 4\sin \frac{t_{1}}{2}\sin \frac{t_{2}}{2}\right) ^{2}}%
\left( (1-\cos (3mt_{1}))(1-\cos (3nt_{2}))\right. \\
& \left. -(1-\cos (3mt_{1}))(1-\cos (nt_{2}))-(1-\cos (mt_{1}))(1-\cos
(3nt_{2}))\right. \\
& \left. +(1-\cos (mt_{1}))(1-\cos (nt_{2}))\right) \\
=& \frac{1}{4mn\left( 4\sin \frac{t_{1}}{2}\sin \frac{t_{2}}{2}\right) ^{2}}%
\left( \cos (mt_{1})-\cos (3mt_{1})\right) \left( \cos (nt_{2})-\cos
(3nt_{2})\right) \\
=& \frac{S(t_{1},t_{2})}{mn\left( 4\sin \frac{t_{1}}{2}\sin \frac{t_{2}}{2}%
\right) ^{2}},
\end{split}
\label{eq6}
\end{equation}%
where%
\begin{equation*}
S(t_{1},t_{2}):=\sin 2mt_{1}\sin mt_{1}\sin 2nt_{2}\sin nt_{2}.
\end{equation*}
Therefore, using (\ref{eq5}) and (\ref{eq6}), we get the second type Double
Delayed Arithmetic Mean as 
\begin{equation*}
\sigma_{m,2m;n,2n}(f;x,y)=\frac{1}{mn\pi^2}\int_{0}^{\pi}\int_{0}^{%
\pi}h_{x,y}(t_1,t_2)\frac{S(t_1,t_2)}{\left(4\sin \frac{t_1}{2}\sin \frac{t_2%
}{2}\right)^2}dt_1dt_2.
\end{equation*}

Throughout this paper we agree to put: $h_1:=h_1(m)=\frac{\pi}{m}$, $%
h_2:=h_2(n)=\frac{\pi}{n}$,
\begin{align*}
\phi_{x,y}(t_1,t_2) & := \frac{1}{4} ( f(x+t_1,y+t_2) + f(x-t_1,y+t_2) \\
& \hspace{1cm} + f(x+t_1,y-t_2) + f(x-t_1,y-t_2) - 4 f(x,y) ), \\
H_{x,z_1,y,z_2}(t_1,t_2) & :=\phi_{x+z_1,y+z_2}(t_1,t_2)-\phi_{x,y}(t_1,t_2),
\end{align*}
and 
\begin{align*}
D_{m,n}(x,y):=\sigma_{m,2m;n,2n}(f;x,y)-f(x,y).
\end{align*}

In order to achieve to our aim, we need some helpful lemmas given in next section. 


\section{Auxiliary Results}

\begin{lemma}
(The generalized Minkowski inequality \cite[p. 21]{SMN})\label{le1} For a
function $g(x_{1},y_{1})$ given on a measurable set $E:=E_{1}\times
E_{2}\subset \mathbb{R}_{2}$, where $x_{1}=(x_{1},x_{2},\dots ,x_{m})$ and $%
y_{1}=(x_{m+1},x_{m+2},\dots ,x_{n})$, the following inequality holds: 
\begin{equation*}
\left( \int_{E_{1}}\left\vert \int_{E_{2}}g(x_{1},y_{1})dy_{1}\right\vert
^{p}dx_{1}\right) ^{\frac{1}{p}}\leq \int_{E_{2}}\left(
\int_{E_{1}}\left\vert g(x_{1},y_{1})\right\vert ^{p}dy_{1}\right) ^{\frac{1%
}{p}}dx_{1},
\end{equation*}%
for those values of $p$ for which the right-hand side of this inequality is
finite.
\end{lemma}

\begin{lemma}
\label{le2} Let $\omega_1$, $\omega_2$, $v_1$ and $v_2$ be moduli of
continuity so that $\frac{\omega_1 (t_1)}{v_1(t_1)}$ is non-decreasing in $%
t_1$, $\frac{\omega_2 (t_2)}{v_2(t_2)}$ is non-decreasing in $t_2$, and $%
f\in H_{p}^{(\omega_1, \omega_2)}$. Then for $0< t_1\leq \pi$, $0< t_2\leq
\pi$, and $p\geq 1$,
\begin{enumerate}
\item[(i)] $\|\phi_{x,y}(t_1,t_2)\|_p=\mathcal{O}\left(\omega_1
(t_1)+\omega_2 (t_2)\right)$,

\item[(ii)] $\|H_{x,z_1,y,z_2}(t_1,t_2)\|_p=\mathcal{O}\left(\omega_1
(t_1)+\omega_2 (t_2)\right)$, \newline
$\|H_{x,z_1,y,z_2}(t_1,t_2)\|_p=\mathcal{O}\left(\omega_1 (|z_1|)+\omega_2
(|z_2|)\right)$, \newline
$\|H_{x,z_1,y,z_2}(t_1+h_1,t_2)\|_p=\mathcal{O}\left(\omega_1
(t_1+h_1)+\omega_2 (t_2)\right)$, \newline
$\|H_{x,z_1,y,z_2}(t_1+h_1,t_2)\|_p=\mathcal{O}\left(\omega_1
(|z_1|)+\omega_2 (|z_2|)\right)$, \newline
$\|H_{x,z_1,y,z_2}(t_1,t_2+h_2)\|_p=\mathcal{O}\left(\omega_1 (t_1)+\omega_2
(t_2+h_2)\right)$, \newline
$\|H_{x,z_1,y,z_2}(t_1,t_2+h_2)\|_p=\mathcal{O}\left(\omega_1
(|z_1|)+\omega_2 (|z_2|)\right)$, \newline
$\|H_{x,z_1,y,z_2}(t_1+h_1,t_2+h_2)\|_p=\mathcal{O}\left(\omega_1
(t_1+h_1)+\omega_2 (t_2+h_2)\right)$, \newline
$\|H_{x,z_1,y,z_2}(t_1+h_1,t_2+h_2)\|_p=\mathcal{O}\left(\omega_1
(|z_1|)+\omega_2 (|z_2|)\right)$.
\end{enumerate}
Moreover, if $\omega_1 = \omega_2$ and $v_1=v_2$, then
\begin{enumerate}
\item[(iii)] $\|H_{x,z_1,y,z_2}(t_1,t_2)\|_p=\mathcal{O}\left( (v_1 (|z_1|)
+ (v_1 (|z_2|)) \left( \frac{\omega_1 (t_1)}{v_1 (t_1)} + \frac{\omega_1
(t_2)}{v_1 (t_2)}\right) \right)$, \newline
$\|H_{x,z_1,y,z_2}(t_1+h_1,t_2)\|_p=\mathcal{O}\left( (v_1 (|z_1|) + (v_1
(|z_2|)) \left( \frac{\omega_1 (t_1+h_1)}{v_1 (t_1+h_1)} + \frac{\omega_1
(t_2)}{v_1 (t_2)}\right) \right)$, \newline
$\|H_{x,z_1,y,z_2}(t_1,t_2+h_2)\|_p=\mathcal{O}\left( (v_1 (|z_1|) + (v_1
(|z_2|)) \left( \frac{\omega_1 (t_1)}{v_1 (t_1)} + \frac{\omega_1 (t_2+h_2)}{%
v_1 (t_2+h_2)}\right) \right)$, \newline
$\|H_{x,z_1,y,z_2}(t_1+h_1,t_2+h_2)\|_p=\mathcal{O}\left( (v_1 (|z_1|) +
(v_1 (|z_2|)) \left( \frac{\omega_1 (t_1+h_1)}{v_1 (t_1+h_1)} + \frac{%
\omega_1 (t_2+h_2)}{v_1 (t_2+h_2)}\right) \right)$.

\item[(iv)] $\|H_{x,z_1,y,z_2}(t_1,t_2)-H_{x,z_1,y,z_2}(t_1+h_1,t_2)\|_p 
\newline
=\mathcal{O}\left( (v_1 (|z_1|) + (v_1 (|z_2|)) \left( \frac{\omega_1 (h_1)}{%
v_1 (h_1)} + \frac{\omega_1 (h_2)}{v_1 (h_2)}\right) \right)$, \newline
$\|H_{x,z_1,y,z_2}(t_1,t_2)-H_{x,z_1,y,z_2}(t_1,t_2+h_2)\|_p\newline
=\mathcal{O}\left( (v_1 (|z_1|) + (v_1 (|z_2|)) \left( \frac{\omega_1 (h_1)}{%
v_1 (h_1)} + \frac{\omega_1 (h_2)}{v_1 (h_2)}\right) \right)$, \newline
$\|H_{x,z_1,y,z_2}(t_1,t_2+h_2)-H_{x,z_1,y,z_2}(t_1+h_1,t_2+h_2)\|_p\newline
=\mathcal{O}\left( (v_1 (|z_1|) + (v_1 (|z_2|)) \left( \frac{\omega_1 (h_1)}{%
v_1 (h_1)} + \frac{\omega_1 (h_2)}{v_1 (h_2)}\right) \right)$, \newline
$\|H_{x,z_1,y,z_2}(t_1+h_1,t_2)-H_{x,z_1,y,z_2}(t_1+h_1,t_2+h_2)\|_p\newline
=\mathcal{O}\left( (v_1 (|z_1|) + (v_1 (|z_2|)) \left( \frac{\omega_1 (h_1)}{%
v_1 (h_1)} + \frac{\omega_1 (h_2)}{v_1 (h_2)}\right) \right)$, \newline
$\|H_{x,z_1,y,z_2}(t_1,t_2)-H_{x,z_1,y,z_2}(t_1+h_1,t_2) -
H_{x,z_1,y,z_2}(t_1,t_2+h_2) + H_{x,z_1,y,z_2}(t_1+h_1,t_2+h_2) \|_p=%
\mathcal{O}\left( (v_1 (|z_1|) + (v_1 (|z_2|)) \left( \frac{\omega_1 (h_1)}{%
v_1 (h_1)} + \frac{\omega_1 (h_2)}{v_1 (h_2)}\right) \right)$.
\end{enumerate}
\end{lemma}

\begin{proof}
Part (i). We calculate as 
\begin{equation*}
\begin{split}
\|\phi_{x,y}(t_1,t_2)\|_p&\leq \frac{1}{4}\left\{\|f(x+t_1,y+t_2)-f(x,y)%
\|_p+\|f(x,y)-f(x-t_1,y+t_2)\|_p \right. \\
&\quad + \left.
\|f(x,y)-f(x+t_1,y-t_2)\|_p+\|f(x,y)-f(x-t_1,y-t_2)\|_p\right\} \\
&=\mathcal{O}\left(\omega_1 (t_1)+\omega_2 (t_2)\right).
\end{split}%
\end{equation*}

Part (ii). Since $f\in \text{Lip}(\omega_1,\omega_2,p)$, we have 
\begin{equation*}
\begin{split}
\|H_{x,z_1,y,z_2}(t_1,t_2)\|_p&=
\|\phi_{x+z_1,y+z_2}(t_1,t_2)-\phi_{x,y}(t_1,t_2)\|_p \\
&\leq\frac{1}{4}\left\{\|f(x+z_1+t_1,y+z_2+t_2)-f(x+z_1,y+z_2)\|_p \right. \\
&\quad +\|f(x+z_1-t_1,y+z_2+t_2)-f(x+z_1,y+z_2)\|_p \\
&\quad +\|f(x+z_1+t_1,y+z_2-t_2)-f(x+z_1,y+z_2)\|_p \\
&\quad +\|f(x+z_1-t_1,y+z_2-t_2)-f(x+z_1,y+z_2)\|_p \\
&\quad +\|f(x+t_1,y+t_2)-f(x,y)\|_p \\
&\quad +\|f(x-t_1,y+t_2)-f(x,y)\|_p \\
&\quad +\|f(x+t_1,y-t_2)-f(x,y)\|_p \\
&\quad + \left. \|f(x-t_1,y-t_2)-f(x,y)\|_p \right\} \\
&=\mathcal{O}\left(\omega_1 (t_1)+\omega_2 (t_2)\right).
\end{split}%
\end{equation*}

For the second part, a similar reasoning yields 
\begin{equation}  \label{eq9}
\begin{split}
\|H_{x,z_1,y,z_2}(t_1,t_2)\|_p&=
\|\phi_{x+z_1,y+z_2}(t_1,t_2)-\phi_{x,y}(t_1,t_2)\|_p \\
&\leq\frac{1}{4}\left\{\|f(x+z_1+t_1,y+z_2+t_2)-f(x+t_1,y+t_2)\|_p \right. \\
&\quad +\|f(x+z_1-t_1,y+z_1+t_2)-f(x-t_1,y+t_2)\|_p \\
&\quad +\|f(x+z_1+t_1,y+z_2-t_2)-f(x+t_1,y-t_2)\|_p \\
&\quad + \left. \|f(x+z_1-t_1,y+z_2-t_2)-f(x-t_1,y-t_2)\|_p\right\} \\
&\quad +\|f(x+z_1,y+z_2)-f(x,y)\|_p \\
& =\mathcal{O}\left(\omega_1 (|z_1|)+\omega_2 (|z_2|)\right).
\end{split}%
\end{equation}
The other relations can be verified in a very same way. We omit their
proofs.

Part (iii). Using part (ii) and the fact $v_1(t_1)$ is non-decreasing, in
case of $t_1\leq |z_1|$ and $t_2\leq |z_2|$, we get 
\begin{equation*}
\begin{split}
\|H_{x,z_1,y,z_2}(t_1,t_2)\|_p&=\mathcal{O}\left(\omega_1 (t_1)+\omega_1
(t_2)\right) \\
&=\mathcal{O}\left( v_1 (t_1) \frac{\omega_1 (t_1)}{v_1 (t_1)} + v_1 (t_2)%
\frac{\omega_1 (t_2)}{v_1 (t_2)}\right) \\
&=\mathcal{O}\left( v_1 (|z_1|) \frac{\omega_1 (t_1)}{v_1 (t_1)} + v_1
(|z_2|)\frac{\omega_1 (t_2)}{v_1 (t_2)}\right).
\end{split}%
\end{equation*}
Since $\frac{\omega_1 (t_1)}{v_1(t_1)}$ is non-decreasing, and we have the
second part of (ii), then for $t_1\geq |z_1|$ and $t_2\geq |z_2|$, we also
obtain 
\begin{equation*}
\begin{split}
\|H_{x,z_1,y,z_2}(t_1,t_2)\|_p&=\mathcal{O}\left(\omega_1 (|z_1|)+\omega_1
(|z_2|)\right) \\
&=\mathcal{O}\left(v_1 (|z_1|) \frac{\omega_1 (|z_1|)}{v_1 (|z_1|)} + v_1
(|z_2|)\frac{\omega_1 (|z_2|)}{v_1 (|z_2|)}\right) \\
&=\mathcal{O}\left(v_1 (|z_1|) \frac{\omega_1 (t_1)}{v_1 (t_1)} + v_1 (|z_2|)%
\frac{\omega_1 (t_2)}{v_1 (t_2)}\right).
\end{split}%
\end{equation*}
For $t_1\leq |z_1|$ and $t_2\geq |z_2|$, consider two possibilities. If $%
|z_1| \geq t_2$, then 
\begin{equation*}
\begin{split}
\|H_{x,z_1,y,z_2}(t_1,t_2)\|_p&=\mathcal{O}\left(\omega_1 (t_1)+\omega_1
(t_2)\right) \\
&=\mathcal{O}\left( v_1 (t_1) \frac{\omega_1 (t_1)}{v_1 (t_1)} + v_1 (t_2) 
\frac{\omega_1 (t_2)}{v_1 (t_2)} \right) \\
& = \mathcal{O}\left( v_1 (|z_1|) \left( \frac{\omega_1 (t_1)}{v_1 (t_1)} + 
\frac{\omega_1 (t_2)}{v_1 (t_2)}\right) \right).
\end{split}%
\end{equation*}
If $t_2 \geq |z_1|$, then 
\begin{equation*}
\begin{split}
\|H_{x,z_1,y,z_2}(t_1,t_2)\|_p&=\mathcal{O}\left(\omega_1 (|z_1|)+\omega_1
(|z_2|)\right) \\
&=\mathcal{O}\left(v_1 (|z_1|) \frac{\omega_1 (|z_1|)}{v_1 (|z_1|)} + v_1
(|z_2|)\frac{\omega_1 (|z_2|)}{v_1 (|z_2|)}\right) \\
&=\mathcal{O}\left( (v_1 (|z_1|) + v_1 (|z_2|)) \frac{\omega_1 (t_2)}{v_1
(t_2)}\right).
\end{split}%
\end{equation*}
For $t_1\geq |z_1|$ and $t_2\leq |z_2|$, we can get 
\begin{equation*}
\|H_{x,z_1,y,z_2}(t_1,t_2)\|_p = \mathcal{O}\left( v_1 (|z_2|) \left( \frac{%
\omega_1 (t_1)}{v_1 (t_1)} + \frac{\omega_1 (t_2)}{v_1 (t_2)}\right) \right)
\end{equation*}
in case of $|z_2| \geq t_1$, and 
\begin{equation*}
\|H_{x,z_1,y,z_2}(t_1,t_2)\|_p=\mathcal{O}\left( (v_1 (|z_1|) + v_1 (|z_2|)) 
\frac{\omega_1 (t_1)}{v_1 (t_1)}\right)
\end{equation*}
in case of $t_1 \geq |z_2|$ similarly as before. This shows the validity of
the first inequality in part (iii). The other relations can be verified in a
very same way.

Part (iv). We have
\begin{equation*}
\begin{split}
&\|H_{x,z_1,y,z_2}(t_1,t_2)-H_{x,z_1,y,z_2}(t_1+h_1,t_2)\|_p \\
&\leq\|\phi_{x,y}(t_1+h_1,t_2)-\phi_{x,y}(t_1,t_2)\|_p \\
&\quad +\|\phi_{x+z_1,y+z_2}(t_1+h_1,t_2)-\phi_{x+z_1,y+z_2}(t_1,t_2)\|_p \\
&=\mathcal{O}\left(\omega_1 (h_1)\right) = \mathcal{O}\left(\omega_1 (h_1) +
\omega_1 (h_2)\right)
\end{split}%
\end{equation*}
and 
\begin{equation*}
\begin{split}
&\|H_{x,z_1,y,z_2}(t_1,t_2)-H_{x,z_1,y,z_2}(t_1+h_1,t_2)\|_p \\
&\leq\|H_{x,z_1,y,z_2}(t_1,t_2)\|_p + \|H_{x,z_1,y,z_2}(t_1+h_1,t_2)\|_p \\
&=\mathcal{O}\left(\omega_1 (|z_1|)+\omega_2 (|z_2|)\right),
\end{split}%
\end{equation*}
while 
\begin{equation*}
\begin{split}
&\|H_{x,z_1,y,z_2}(t_1,t_2)-H_{x,z_1,y,z_2}(t_1,t_2+h_2)\|_p \\
&\leq\|\phi_{x,y}(t_1,t_2+h_2)-\phi_{x,y}(t_1,t_2)\|_p \\
&\quad +\|\phi_{x+z_1,y+z_2}(t_1,t_2+h_2)-\phi_{x+z_1,y+z_2}(t_1,t_2)\|_p \\
&=\mathcal{O}\left(\omega_1 (h_2)\right) = \mathcal{O}\left(\omega_1 (h_1) +
\omega_1 (h_2)\right)
\end{split}%
\end{equation*}
and 
\begin{equation*}
\begin{split}
&\|H_{x,z_1,y,z_2}(t_1,t_2)-H_{x,z_1,y,z_2}(t_1,t_2+h_2)\|_p \\
&\leq\|H_{x,z_1,y,z_2}(t_1,t_2)\|_p + \|H_{x,z_1,y,z_2}(t_1,t_2+h_2)\|_p \\
&=\mathcal{O}\left(\omega_1 (|z_1|)+\omega_2 (|z_2|)\right),
\end{split}%
\end{equation*}
furthermore 
\begin{equation*}
\begin{split}
\| H_{x,z_1,y,z_2}(t_1,t_2)- &H_{x,z_1,y,z_2}(t_1+h_1,t_2)-
H_{x,z_1,y,z_2}(t_1,t_2+h_2)\\
&\quad + H_{x,z_1,y,z_2}(t_1+h_1,t_2+h_2) \|_p \\
&\leq \|H_{x,z_1,y,z_2}(t_1,t_2)-H_{x,z_1,y,z_2}(t_1+h_1,t_2)\|_p \\
&\quad + \| H_{x,z_1,y,z_2}(t_1,t_2+h_2) -
H_{x,z_1,y,z_2}(t_1+h_1,t_2+h_2)\|_p \\
&= \mathcal{O}\left(\omega_1 (h_1) + \omega_1 (h_2)\right)
\end{split}%
\end{equation*}
and 
\begin{equation*}
\begin{split}
\|H_{x,z_1,y,z_2}(t_1,t_2)&-H_{x,z_1,y,z_2}(t_1+h_1,t_2) -
H_{x,z_1,y,z_2}(t_1,t_2+h_2)\\
&\quad + H_{x,z_1,y,z_2}(t_1+h_1,t_2+h_2)\|_p \\
&\leq\|H_{x,z_1,y,z_2}(t_1,t_2)\|_p + \|H_{x,z_1,y,z_2}(t_1+h_1,t_2)\|_p \\
&\quad + \| H_{x,z_1,y,z_2}(t_1,t_2+h_2)\|_p + \|
H_{x,z_1,y,z_2}(t_1+h_1,t_2+h_2)\|_p \\
&=\mathcal{O}\left(\omega_1 (|z_1|)+\omega_2 (|z_2|)\right).
\end{split}%
\end{equation*}
From analogous estimations, we can obtain part (iv) by considering the four
cases $h_1\leq |z_1|$ and $h_2\leq |z_2|$; $h_1\geq |z_1|$ and $h_2\geq
|z_2| $; $h_1\leq |z_1|$ and $h_2\geq |z_2|$; $h_1\geq |z_1|$ and $h_2\leq
|z_2|$, respectively, as in part (iii). We omit the details.
\end{proof}

\section{Main Results}

We prove the following statement. 

\begin{theorem}
\label{the01} Let $\omega $ and $v$ be moduli of continuity so that $\frac{%
\omega (t)}{v(t)}$ is non-decreasing in $t$. If $f\in H_{p}^{(\omega ,\omega
)}$, $p\geq 1$, then 
\begin{equation*}
\Vert \sigma _{m,2m;n,2n}(f)-f\Vert _{p}^{(v,v)}= \mathcal{O} \left( \frac{\omega (h_{1})%
}{v(h_{1})}+\frac{\omega (h_{2})}{v(h_{2})}\right) ,
\end{equation*}%
where $h_{1}=\frac{\pi }{m}$, $h_{2}=\frac{\pi }{n}$ for $m,n\in 
\mathbb{N}
$.
\end{theorem}

\allowdisplaybreaks{\begin{proof}
Using the equality 
\begin{equation}
\int_{0}^{\pi }\int_{0}^{\pi }\frac{\sin 2mt_{1}\sin mt_{1}\sin 2nt_{2}\sin
nt_{2}}{\left( 4\sin \frac{t_{1}}{2}\sin \frac{t_{2}}{2}\right) ^{2}}%
dt_{1}dt_{2}=\frac{mn\pi ^{2}}{4}  \label{eq11}
\end{equation}%
we have 
\begin{equation}
\begin{split}
D_{m,n}(x,y)& =\sigma _{m,2m;n,2n}(f;x,y)-f(x,y) \\
& =\frac{4}{mn\pi ^{2}}\int_{0}^{\pi }\int_{0}^{\pi }\phi _{x,y}(t_{1},t_{2})%
\frac{S(t_{1},t_{2})}{\left( 4\sin \frac{t_{1}}{2}\sin \frac{t_{2}}{2}%
\right) ^{2}}dt_{1}dt_{2},
\end{split}
\label{eq12}
\end{equation}%
where 
\begin{equation*}
\begin{split}
\phi _{x,y}(t_{1},t_{2})=& \frac{1}{4}\big[%
f(x+t_{1},y+t_{2})+f(x-t_{1},y+t_{2}) \\
& \quad +f(x+t_{1},y-t_{2})+f(x-t_{1},y-t_{2})-4f(x,y)\big]
\end{split}%
\end{equation*}%
and 
\begin{equation*}
S(t_{1},t_{2})=\sin 2mt_{1}\sin mt_{1}\sin 2nt_{2}\sin nt_{2}.
\end{equation*}
By definition, we have 
\begin{equation}
\Vert D_{m,n}\Vert _{p}^{(v,v)}:=\Vert D_{m,n}\Vert _{p}+\sup_{z_{1}\neq
0,\,\,z_{2}\neq 0}\frac{\Vert D_{m,n}(x+z_{1},y+z_{2})-D_{m,n}(x,y)\Vert _{p}%
}{v(|z_{1}|)+v(|z_{2}|)}.  \label{eq13}
\end{equation}%
Now, we can write 
\begin{equation}
\begin{split}
& D_{m,n}(x+z_{1},y+z_{2})-D_{m,n}(x,y) \\
& =\frac{4}{mn\pi ^{2}}\int_{0}^{\pi }\int_{0}^{\pi }\left[ \phi
_{x+z_{1},y+z_{2}}(t_{1},t_{2})-\phi _{x,y}(t_{1},t_{2})\right] \frac{%
S(t_{1},t_{2})}{\left( 4\sin \frac{t_{1}}{2}\sin \frac{t_{2}}{2}\right) ^{2}}%
dt_{1}dt_{2} \\
& =\frac{4}{mn\pi ^{2}}\int_{0}^{\pi }\int_{0}^{\pi
}H_{x,z_{1},y,z_{2}}(t_{1},t_{2})\frac{S(t_{1},t_{2})}{\left( 4\sin \frac{%
t_{1}}{2}\sin \frac{t_{2}}{2}\right) ^{2}}dt_{1}dt_{2} \\
& =\frac{4}{mn\pi ^{2}}\int_{0}^{h_{1}}%
\int_{0}^{h_{2}}H_{x,z_{1},y,z_{2}}(t_{1},t_{2})\frac{S(t_{1},t_{2})}{\left(
4\sin \frac{t_{1}}{2}\sin \frac{t_{2}}{2}\right) ^{2}}dt_{1}dt_{2} \\
& \quad +\frac{4}{mn\pi ^{2}}\int_{h_{1}}^{\pi
}\int_{0}^{h_{2}}H_{x,z_{1},y,z_{2}}(t_{1},t_{2})\frac{S(t_{1},t_{2})}{%
\left( 4\sin \frac{t_{1}}{2}\sin \frac{t_{2}}{2}\right) ^{2}}dt_{1}dt_{2} \\
& \quad +\frac{4}{mn\pi ^{2}}\int_{0}^{h_{1}}\int_{h_{2}}^{\pi
}H_{x,z_{1},y,z_{2}}(t_{1},t_{2})\frac{S(t_{1},t_{2})}{\left( 4\sin \frac{%
t_{1}}{2}\sin \frac{t_{2}}{2}\right) ^{2}}dt_{1}dt_{2} \\
& \quad +\frac{4}{mn\pi ^{2}}\int_{h_{1}}^{\pi }\int_{h_{2}}^{\pi
}H_{x,z_{1},y,z_{2}}(t_{1},t_{2})\frac{S(t_{1},t_{2})}{\left( 4\sin \frac{%
t_{1}}{2}\sin \frac{t_{2}}{2}\right) ^{2}}dt_{1}dt_{2}:=\sum_{r=1}^{4}J_{r},
\end{split}
\label{eq14}
\end{equation}

Using Lemma \ref{le1} for $p\geq 1$, by the estimate 
\begin{equation*}
\left\vert \frac{S(t_{1},t_{2})}{\left( 4\sin \frac{t_{1}}{2}\sin \frac{t_{2}%
}{2}\right) ^{2}}\right\vert =\mathcal{O}\left( m^{2}n^{2}\right) ,\quad
0<t_{1}\leq \pi ,0<t_{2}\leq \pi ,
\end{equation*}%
and Lemma \ref{le2}, we have 
\begin{equation}
\begin{split}
\Vert J_{1}\Vert _{p}& \leq \frac{4}{mn\pi ^{2}}\int_{0}^{h_{1}}%
\int_{0}^{h_{2}}\Vert H_{x,z_{1},y,z_{2}}(t_{1},t_{2})\Vert _{p}\left\vert 
\frac{S(t_{1},t_{2})}{\left( 4\sin \frac{t_{1}}{2}\sin \frac{t_{2}}{2}%
\right) ^{2}}\right\vert dt_{1}dt_{2} \\
& =\mathcal{O}\left( \left( v(|z_{1}|)+v(|z_{2}|)\right)
mn\int_{0}^{h_{1}}\int_{0}^{h_{2}}\left( \frac{\omega (t_{1})}{v(t_{1})}+%
\frac{\omega (t_{2})}{v(t_{2})}\right) dt_{1}dt_{2}\right) \\
& =\mathcal{O}\left( \left( v(|z_{1}|)+v(|z_{2}|)\right) mn\left( \frac{%
\omega (h_{1})}{v(h_{1})}+\frac{\omega (h_{2})}{v(h_{2})}\right)
\int_{0}^{h_{1}}\int_{0}^{h_{2}}dt_{1}dt_{2}\right) \\
& =\mathcal{O}\left( \left( v(|z_{1}|)+v(|z_{2}|)\right) \left( \frac{\omega
(h_{1})}{v(h_{1})}+\frac{\omega (h_{2})}{v(h_{2})}\right) \right) .
\end{split}
\label{eq15}
\end{equation}

The quantity $J_{2}$ can be written as follows: 
\begin{equation}
\begin{split}
J_{2}& =\frac{4}{mn\pi ^{2}}\int_{h_{1}}^{\pi
}\int_{0}^{h_{2}}H_{x,z_{1},y,z_{2}}(t_{1},t_{2})\sin 2mt_{1}\sin mt_{1}\\%
&\quad \times \frac{\sin 2nt_{2}\sin nt_{2}}{\left( 2\sin \frac{t_{2}}{2}\right) ^{2}}%
\left( \frac{1}{\left( 2\sin \frac{t_{1}}{2}\right) ^{2}}-\frac{1}{t_{1}^{2}}%
\right) dt_{1}dt_{2}\\
&\quad +\frac{4}{mn\pi ^{2}}\int_{h_{1}}^{\pi
}\int_{0}^{h_{2}}H_{x,z_{1},y,z_{2}}(t_{1},t_{2})\frac{S(t_{1},t_{2})}{%
t_{1}^{2}\left( 2\sin \frac{t_{2}}{2}\right) ^{2}}%
dt_{1}dt_{2}:=J_{21}+J_{22}.
\end{split}
\label{eq16}
\end{equation}%
Since the function 
\begin{equation*}
\frac{1}{\left( 2\sin \frac{t_{1}}{2}\right) ^{2}}-\frac{1}{t_{1}^{2}}
\end{equation*}%
is bounded for $0<t_{1}\leq \pi $ and 
\begin{equation*}
\left\vert \frac{\sin 2nt_{2}\sin nt_{2}}{\left( 2\sin \frac{t_{2}}{2}%
\right) ^{2}}\right\vert =\mathcal{O}\left( n^{2}\right) ,\quad 0<t_{2}\leq
\pi ,
\end{equation*}%
then by Lemma \ref{le1} with $p\geq 1$ and Lemma \ref{le2}, we get 
\begin{equation}
\begin{split}
\Vert J_{21}\Vert _{p}& =\mathcal{O}\left( \frac{n}{m}\right)
\int_{h_{1}}^{\pi }\int_{0}^{h_{2}}\Vert
H_{x,z_{1},y,z_{2}}(t_{1},t_{2})\Vert _{p}dt_{1}dt_{2} \\
& =\mathcal{O}\left( \left( v(|z_{1}|)+v(|z_{2}|)\right) \frac{n}{m}%
\int_{h_{1}}^{\pi }\int_{0}^{h_{2}}\left( \frac{\omega (t_{1})}{v(t_{1})}+%
\frac{\omega (t_{2})}{v(t_{2})}\right) dt_{1}dt_{2}\right) \\
& =\mathcal{O}\left( \left( v(|z_{1}|)+v(|z_{2}|)\right) \frac{n}{m}\left( 
\frac{\omega (\pi )}{v(\pi )}+\frac{\omega (h_{2})}{v(h_{2})}\right)
\int_{h_{1}}^{\pi }\int_{0}^{h_{2}}dt_{1}dt_{2}\right) \\
& =\mathcal{O}\left( \frac{1}{m}\left( v(|z_{1}|)+v(|z_{2}|)\right) \left( 
\frac{\omega (\pi )}{v(\pi )}+\frac{\omega (h_{2})}{v(h_{2})}\right) \right)
.
\end{split}
\label{eq17}
\end{equation}%
It is clear that for $m\in 
\mathbb{N}
$%
\begin{equation*}
\frac{1}{m}\omega \left( \pi \right) \leq 2\omega \left( \frac{\pi }{m}%
\right) ,
\end{equation*}%
whence 
\begin{equation*}
\left\Vert J_{21}\right\Vert _{p}=O\left( \left(
v(|z_{1}|)+v(|z_{2}|)\right) \left( \frac{\omega (h_{1})}{v(h_{1})}+\frac{%
\omega (h_{2})}{v(h_{2})}\right) \right) .
\end{equation*}%
Further, substituting $t_{1}+h_{1}$ in place of $t_{1}$ in 
\begin{equation}
J_{22}=\frac{4}{mn\pi ^{2}}\int_{h_{1}}^{\pi
}\int_{0}^{h_{2}}H_{x,z_{1},y,z_{2}}(t_{1},t_{2})\frac{S(t_{1},t_{2})}{%
t_{1}^{2}\left( 2\sin \frac{t_{2}}{2}\right) ^{2}}dt_{1}dt_{2}  \label{eq18}
\end{equation}%
we obtain 
\begin{equation}
J_{22}=-\frac{4}{mn\pi ^{2}}\int_{0}^{\pi
-h_{1}}\int_{0}^{h_{2}}H_{x,z_{1},y,z_{2}}(t_{1}+h_{1},t_{2})\frac{%
S(t_{1},t_{2})}{(t_{1}+h_{1})^{2}\left( 2\sin \frac{t_{2}}{2}\right) ^{2}}%
dt_{1}dt_{2},  \label{eq19}
\end{equation}%
Hence, from (\ref{eq18}) and (\ref{eq19}) we get 
\begin{equation*}
\begin{split}
J_{22}& =\frac{2}{mn\pi ^{2}}\int_{h_{1}}^{\pi
}\int_{0}^{h_{2}}H_{x,z_{1},y,z_{2}}(t_{1},t_{2})\frac{S(t_{1},t_{2})}{%
t_{1}^{2}\left( 2\sin \frac{t_{2}}{2}\right) ^{2}}dt_{1}dt_{2} \\
& \quad -\frac{2}{mn\pi ^{2}}\int_{0}^{\pi
-h_{1}}\int_{0}^{h_{2}}H_{x,z_{1},y,z_{2}}(t_{1}+h_{1},t_{2})\frac{%
S(t_{1},t_{2})}{(t_{1}+h_{1})^{2}\left( 2\sin \frac{t_{2}}{2}\right) ^{2}}%
dt_{1}dt_{2}\\
& =\frac{2}{mn\pi ^{2}}\int_{h_{1}}^{\pi
}\int_{0}^{h_{2}}H_{x,z_{1},y,z_{2}}(t_{1},t_{2})\frac{S(t_{1},t_{2})}{%
t_{1}^{2}\left( 2\sin \frac{t_{2}}{2}\right) ^{2}}dt_{1}dt_{2} \\
& \quad -\frac{2}{mn\pi ^{2}}\int_{0}^{h_{1}}%
\int_{0}^{h_{2}}H_{x,z_{1},y,z_{2}}(t_{1}+h_{1},t_{2})\frac{S(t_{1},t_{2})}{%
(t_{1}+h_{1})^{2}\left( 2\sin \frac{t_{2}}{2}\right) ^{2}}dt_{1}dt_{2} \\
& \quad -\frac{2}{mn\pi ^{2}}\int_{h_{1}}^{\pi
}\int_{0}^{h_{2}}H_{x,z_{1},y,z_{2}}(t_{1}+h_{1},t_{2})\frac{S(t_{1},t_{2})}{%
(t_{1}+h_{1})^{2}\left( 2\sin \frac{t_{2}}{2}\right) ^{2}}dt_{1}dt_{2} \\
& \quad +\frac{2}{mn\pi ^{2}}\int_{\pi -h_{1}}^{\pi
}\int_{0}^{h_{2}}H_{x,z_{1},y,z_{2}}(t_{1}+h_{1},t_{2})\frac{S(t_{1},t_{2})}{%
(t_{1}+h_{1})^{2}\left( 2\sin \frac{t_{2}}{2}\right) ^{2}}dt_{1}dt_{2} 
\end{split}
\end{equation*}

\begin{equation}
\begin{split}
& =\frac{2}{mn\pi ^{2}}\int_{h_{1}}^{\pi }\int_{0}^{h_{2}}\left( \frac{%
H_{x,z_{1},y,z_{2}}(t_{1},t_{2})}{t_{1}^{2}}-\frac{%
H_{x,z_{1},y,z_{2}}(t_{1}+h_{1},t_{2})}{(t_{1}+h_{1})^{2}}\right) \frac{S(t_{1},t_{2})}{\left( 2\sin \frac{t_{2}}{2}%
\right) ^{2}}dt_{1}dt_{2} \\
& \quad -\frac{2}{mn\pi ^{2}}\int_{0}^{h_{1}}%
\int_{0}^{h_{2}}H_{x,z_{1},y,z_{2}}(t_{1}+h_{1},t_{2})\frac{S(t_{1},t_{2})}{%
(t_{1}+h_{1})^{2}\left( 2\sin \frac{t_{2}}{2}\right) ^{2}}dt_{1}dt_{2} \\
& \quad +\frac{2}{mn\pi ^{2}}\int_{\pi -h_{1}}^{\pi
}\int_{0}^{h_{2}}H_{x,z_{1},y,z_{2}}(t_{1}+h_{1},t_{2})\frac{S(t_{1},t_{2})}{%
(t_{1}+h_{1})^{2}\left( 2\sin \frac{t_{2}}{2}\right) ^{2}}%
dt_{1}dt_{2}:=\sum_{s=1}^{3}J_{22}^{(s)}.
\end{split}
\label{eq20}
\end{equation}%
Using the inequalities $\frac{2}{\pi }\beta \leq \sin \beta $ for $\beta \in
(0,\pi /2)$, $\sin \beta \leq \beta $ for $\beta \in (0,\pi )$, Lemma \ref%
{le1} and Lemma \ref{le2}, we have 
\begin{equation}
\begin{split}
\Vert J_{22}^{(2)}\Vert _{p}& =\mathcal{O}\left( \frac{1}{mn}\right)
\int_{0}^{h_{1}}\int_{0}^{h_{2}}\Vert
H_{x,z_{1},y,z_{2}}(t_{1}+h_{1},t_{2})\Vert _{p}\frac{(mt_{1}nt_{2})^{2}}{%
(t_{1}+h_{1})^{2}\left( \frac{t_{2}}{\pi }\right) ^{2}}dt_{1}dt_{2} \\
& =\mathcal{O}\left( mn\right)
\int_{0}^{h_{1}}\int_{0}^{h_{2}}(v(|z_{1}|)+(v(|z_{2}|))\left( \frac{\omega
(t_{1}+h_{1})}{v(t_{1}+h_{1})}+\frac{\omega (h_{2})}{v(h_{2})}\right)
dt_{1}dt_{2} \\
& =\mathcal{O}\left( mn\right) (v(|z_{1}|)+v(|z_{2}|))\left( \frac{\omega
(2h_{1})}{v(2h_{1})}+\frac{\omega (h_{2})}{v(h_{2})}\right)
\int_{0}^{h_{1}}\int_{0}^{h_{2}}dt_{1}dt_{2} \\
& =\mathcal{O}\left( mn\right) (v(|z_{1}|)+v(|z_{2}|))\left( \frac{\omega
(h_{1})}{v(h_{1})}+\frac{\omega (h_{2})}{v(h_{2})}\right) h_{1}h_{2} \\
& =\mathcal{O}\left( (v(|z_{1}|)+v(|z_{2}|))\left( \frac{\omega (h_{1})}{%
v(h_{1})}+\frac{\omega (h_{2})}{v(h_{2})}\right) \right) .
\end{split}
\label{eq21}
\end{equation}%
Moreover, the inequalities $|\sin \beta |\leq 1$ for $\beta \in (\pi
-h_{1},\pi )$, $\sin \beta \leq \beta $ for $\beta \in (0,\pi )$, Lemma \ref%
{le1}, and Lemma \ref{le2} imply 
\begin{equation}
\begin{split}
\Vert J_{22}^{(3)}\Vert _{p}& =\mathcal{O}\left( \frac{1}{mn}\right)
\int_{\pi -h_{1}}^{\pi }\int_{0}^{h_{2}}\Vert
H_{x,z_{1},y,z_{2}}(t_{1}+h_{1},t_{2})\Vert _{p}\frac{(nt_{2})^{2}}{%
(t_{1}+h_{1})^{2}\left( \frac{t_{2}}{\pi }\right) ^{2}}dt_{1}dt_{2} \\
& =\mathcal{O}\left( \frac{n}{m}\right) \int_{\pi -h_{1}}^{\pi
}\!\int_{0}^{h_{2}}(v(|z_{1}|)+v(|z_{2}|))\left( \frac{\omega (t_{1}+h_{1})}{%
v(t_{1}+h_{1})}+\frac{\omega (t_{2})}{v(t_{2})}\right) \frac{1}{%
(t_{1}+h_{1})^{2}}dt_{1}dt_{2} \\
& =\mathcal{O}\left( \frac{n}{m}\right) (v(|z_{1}|)+v(|z_{2}|))\int_{\pi
}^{\pi +h_{1}}\int_{0}^{h_{2}}\left( \frac{\omega (\theta _{1})}{v(\theta
_{1})}+\frac{\omega (t_{2})}{v(t_{2})}\right) \frac{1}{\theta _{1}^{2}}%
d\theta _{1}dt_{2} \\
& =\mathcal{O}\left( \frac{n}{m}\right) (v(|z_{1}|)+v(|z_{2}|))\left( \frac{%
\omega (\pi +h_{1})}{v(\pi )}+\frac{\omega (h_{2})}{v(h_{2})}\right)
\int_{\pi }^{\pi +h_{1}}\int_{0}^{h_{2}}\frac{d\theta _{1}dt_{2}}{\theta
_{1}^{2}} \\
& =\mathcal{O}\left( \frac{n}{m}\right) (v(|z_{1}|)+v(|z_{2}|))\left( \frac{%
\omega (\pi )+\omega (h_{1})}{v(\pi )}+\frac{\omega (h_{2})}{v(h_{2})}%
\right) \frac{h_{1}h_{2}}{\pi (\pi +h_{1})} \\
& =\mathcal{O}\left( \frac{1}{m^{2}}\right) (v(|z_{1}|)+v(|z_{2}|))\left( 
\frac{\omega (\pi )}{v(\pi )}+\frac{\omega (h_{2})}{v(h_{2})}\right) \\
& =\mathcal{O}\left( \frac{1}{m}(v(|z_{1}|)+v(|z_{2}|))\left( \frac{\omega
(h_{1})}{v(h_{1})}+\frac{\omega (h_{2})}{v(h_{2})}\right) \right) .
\end{split}
\label{eq22}
\end{equation}%
For $J_{22}^{(1)}$ we can write 
\begin{equation}
\begin{split}
J_{22}^{(1)}& =\frac{2}{mn\pi ^{2}}\int_{h_{1}}^{\pi }\int_{0}^{h_{2}}\Bigg(%
\frac{H_{x,z_{1},y,z_{2}}(t_{1},t_{2})}{t_{1}^{2}}-\frac{%
H_{x,z_{1},y,z_{2}}(t_{1}+h_{1},t_{2})}{t_{1}^{2}} \\
& \qquad +\frac{H_{x,z_{1},y,z_{2}}(t_{1}+h_{1},t_{2})}{t_{1}^{2}}-\frac{%
H_{x,z_{1},y,z_{2}}(t_{1}+h_{1},t_{2})}{(t_{1}+h_{1})^{2}}\Bigg) \\
& \qquad \times \frac{S(t_{1},t_{2})}{\left( 2\sin \frac{t_{2}}{2}\right)
^{2}}dt_{1}dt_{2} \\
& =\frac{2}{mn\pi ^{2}}\int_{h_{1}}^{\pi }\int_{0}^{h_{2}}\left(
H_{x,z_{1},y,z_{2}}(t_{1},t_{2})-H_{x,z_{1},y,z_{2}}(t_{1}+h_{1},t_{2})%
\right) \\
& \qquad \times \frac{S(t_{1},t_{2})}{t_{1}^{2}\left( 2\sin \frac{t_{2}}{2}%
\right) ^{2}}dt_{1}dt_{2} \\
& \quad +\frac{2}{mn\pi ^{2}}\int_{h_{1}}^{\pi
}\int_{0}^{h_{2}}H_{x,z_{1},y,z_{2}}(t_{1}+h_{1},t_{2})\frac{S(t_{1},t_{2})}{%
\left( 2\sin \frac{t_{2}}{2}\right) ^{2}} \\
& \qquad \times \left( \frac{1}{t_{1}^{2}}-\frac{1}{(t_{1}+h_{1})^{2}}%
\right) dt_{1}dt_{2}:=J_{221}^{(1)}+J_{222}^{(1)}.
\end{split}
\label{eq23}
\end{equation}%
Then, using Lemma \ref{le1}, and Lemma \ref{le2}, we have 
\begin{equation}
\begin{split}
\Vert J_{221}^{(1)}\Vert _{p}& =\mathcal{O}\left( \frac{1}{mn}\right)
\int_{h_{1}}^{\pi }\int_{0}^{h_{2}}\left\Vert
H_{x,z_{1},y,z_{2}}(t_{1},t_{2})-H_{x,z_{1},y,z_{2}}(t_{1}+h_{1},t_{2})%
\right\Vert _{p}\frac{(nt_{2})^{2}}{t_{1}^{2}\left( 2\frac{t_{2}}{\pi }%
\right) ^{2}}dt_{1}dt_{2} \\
& =\mathcal{O}\left( \frac{n}{m}\right) \int_{h_{1}}^{\pi
}\int_{0}^{h_{2}}(v(|z_{1}|)+v(|z_{2}|))\left( \frac{\omega (h_{1})}{v(h_{1})%
}+\frac{\omega (h_{2})}{v(h_{2})}\right) \frac{dt_{1}dt_{2}}{t_{1}^{2}} \\
& =\mathcal{O}\left( \frac{n}{m}\right) (v(|z_{1}|)+v(|z_{2}|))\left( \frac{%
\omega (h_{1})}{v(h_{1})}+\frac{\omega (h_{2})}{v(h_{2})}\right) \frac{h_{2}%
}{h_{1}} \\
& =\mathcal{O}\left( (v(|z_{1}|)+v(|z_{2}|))\left( \frac{\omega (h_{1})}{%
v(h_{1})}+\frac{\omega (h_{2})}{v(h_{2})}\right) \right) .
\end{split}
\label{eq24}
\end{equation}%
With similar reasoning we obtain 
\begin{equation}
\begin{split}
\Vert J_{222}^{(1)}\Vert _{p}& =\mathcal{O}\left( \frac{1}{mn}\right)
\int_{h_{1}}^{\pi }\int_{0}^{h_{2}}\left\Vert
H_{x,z_{1},y,z_{2}}(t_{1}+h_{1},t_{2})\right\Vert _{p}\frac{(nt_{2})^{2}}{%
\left( 2\frac{t_{2}}{\pi }\right) ^{2}}\frac{h_{1}(2t_{1}+h_{1})}{%
t_{1}^{2}(t_{1}+h_{1})^{2}}dt_{1}dt_{2} \\
& =\mathcal{O}\left( \frac{n}{m}\right) \int_{h_{1}}^{\pi
}\int_{0}^{h_{2}}(v(|z_{1}|)+(v(|z_{2}|))\left( \frac{\omega (t_{1}+h_{1})}{%
v(t_{1}+h_{1})}+\frac{\omega (t_{2})}{v(t_{2})}\right) \frac{%
h_{1}dt_{1}dt_{2}}{t_{1}^{2}(t_{1}+h_{1})} \\
& =\mathcal{O}\left( \frac{n}{m^{2}}\right)
(v(|z_{1}|)+v(|z_{2}|))\int_{h_{1}}^{\pi }\left( \frac{\omega (t_{1}+h_{1})}{%
(t_{1}+h_{1})v(h_{1})}+\frac{\omega (h_{2})}{v(h_{2})t_{1}}\right) \frac{%
h_{2}dt_{1}}{t_{1}^{2}} \\
& =\mathcal{O}\left( \frac{1}{m^{2}}\right) \left(
v(|z_{1}|)+v(|z_{2}|)\right) \left( \frac{\omega (2h_{1})}{2h_{1}v(h_{1})}%
\int\limits_{h_{1}}^{\pi }\frac{dt_{1}}{t_{1}^{2}}+\frac{\omega (h_{2})}{%
v(h_{2})}\int\limits_{h_{1}}^{\pi }\frac{dt_{1}}{t_{1}^{3}}\right) \\
& =\mathcal{O}\left( \frac{1}{m^{2}}\right) \left(
v(|z_{1}|)+v(|z_{2}|)\right) \left( m\frac{\omega (h_{1})}{2h_{1}v(t_{1})}%
+m^{2}\frac{\omega (h_{2})}{v(h_{2})}\right) \\
& =\mathcal{O}\left( \left( v(|z_{1}|)+v(|z_{2}|)\right) \left( \frac{\omega
(h_{1})}{v(h_{1})}+\frac{\omega (h_{2})}{v(h_{2})}\right) \right) .
\end{split}
\label{eq25}
\end{equation}%
So, from (\ref{eq23}), (\ref{eq24}) and (\ref{eq25}), we have 
\begin{equation}
\Vert J_{22}^{(1)}\Vert _{p}=\mathcal{O}\left( (v(|z_{1}|)+v(|z_{2}|))\left( 
\frac{\omega (h_{1})}{v(h_{1})}+\frac{\omega (h_{2})}{v(h_{2})}\right)
\right) .  \label{eq26}
\end{equation}%
Now, taking into account (\ref{eq20}), (\ref{eq21}), (\ref{eq22}) and (\ref%
{eq26}), we have 
\begin{equation}
\Vert J_{22}\Vert _{p}=\mathcal{O}\left( \left( v(|z_{1}|)+v(|z_{2}|)\right)
\left( \frac{\omega (h_{1})}{v(h_{1})}+\frac{\omega (h_{2})}{v(h_{2})}%
\right) \right) .  \label{eq27}
\end{equation}%
Whence, using (\ref{eq16}), (\ref{eq17}) and (\ref{eq27}), we obtain 
\begin{equation}
\Vert J_{2}\Vert _{p}=\mathcal{O}\left( \left( v(|z_{1}|)+v(|z_{2}|)\right)
\left( \frac{\omega (h_{1})}{v(h_{1})}+\frac{\omega (h_{2})}{v(h_{2})}%
\right) \right) .  \label{eq28}
\end{equation}

By analogy, we conclude that 
\begin{equation}
\Vert J_{3}\Vert _{p}=\mathcal{O}\left( \left( v(|z_{1}|)+v(|z_{2}|)\right)
\left( \frac{\omega (h_{1})}{v(h_{1})}+\frac{\omega (h_{2})}{v(h_{2})}%
\right) \right) .  \label{eq29}
\end{equation}

Finally, let us estimate the quantity $J_{4}$. Indeed, $J_{4}$ can be
rewritten as 
\begin{align*}
J_{4}& = \frac{4}{mn\pi ^{2}}\int_{h_{1}}^{\pi }\int_{h_{2}}^{\pi
}H_{x,z_{1},y,z_{2}}(t_{1},t_{2})S(t_1,t_2) \\
&\qquad \times \left( \frac{1}{\left( 2\sin \frac{t_{1}}{2}\right) ^{2}}-\frac{1}{%
t_{1}{}^{2}}\right) \left( \frac{1}{\left( 2\sin \frac{t_{2}}{2}\right) ^{2}}%
-\frac{1}{t_{2}{}^{2}}\right) dt_{1}dt_{2}\\
&\quad +\frac{4}{mn\pi ^{2}}\int_{h_{1}}^{\pi }\int_{h_{2}}^{\pi
}H_{x,z_{1},y,z_{2}}(t_{1},t_{2})S(t_1,t_2) \\
& \qquad \times \left( \frac{1}{\left( 2\sin \frac{t_{2}}{2}\right) ^{2}}-\frac{1}{%
t_{2}{}^{2}}\right) \frac{1}{t_{1}{}^{2}}dt_{1}dt_{2}\\
& \quad +\frac{4}{mn\pi ^{2}}\int_{h_{1}}^{\pi }\int_{h_{2}}^{\pi
}H_{x,z_{1},y,z_{2}}(t_{1},t_{2})S(t_1,t_2) \\
& \qquad \times \left( \frac{1}{\left( 2\sin \frac{t_{1}}{2}\right) ^{2}}-\frac{1}{%
t_{1}{}^{2}}\right) \frac{1}{t_{2}{}^{2}}dt_{1}dt_{2}\\
& \quad +\frac{4}{mn\pi ^{2}}\int_{h_{1}}^{\pi }\int_{h_{2}}^{\pi
}H_{x,z_{1},y,z_{2}}(t_{1},t_{2})\frac{S(t_1,t_2)}{\left( t_{1}t_{2}\right)
^{2}}dt_{1}dt_{2} \\
&:= J_{41}+J_{42}+J_{43}+J_{44}.
\end{align*}%
The boundedness of the function 
\begin{equation*}
\left( \frac{1}{\left( 2\sin \frac{t_{1}}{2}\right) ^{2}}-\frac{1}{%
t_{1}{}^{2}}\right) \left( \frac{1}{\left( 2\sin \frac{t_{2}}{2}\right) ^{2}}%
-\frac{1}{t_{2}{}^{2}}\right)
\end{equation*}%
for $0<t_{1},t_{2}\leq \pi $, Lemma \ref{le1} and Lemma \ref{le2}, implies
\begin{align*}
\Vert J_{41}\Vert _{p}&=\mathcal{O}\left( \frac{1}{mn}\right)
\int_{h_{1}}^{\pi }\int_{h_{2}}^{\pi }\Vert
H_{x,z_{1},y,z_{2}}(t_{1},t_{2})\Vert _{p}dt_{1}dt_{2}\\
&=\mathcal{O}\left( \frac{1}{mn}\right) \int_{h_{1}}^{\pi }\int_{h_{2}}^{\pi
}(v(|z_{1}|)+v(|z_{2}|))\left( \frac{\omega (t_{1})}{v(t_{1})}+\frac{\omega
(t_{2})}{v(t_{2})}\right) dt_{1}dt_{2}\\
&=\mathcal{O}\left( \frac{1}{mn}\right) (v(|z_{1}|)+v(|z_{2}|))\frac{\omega
\left( \pi \right) }{v\left( \pi \right) } \\
&=\mathcal{O}\left( \left( v(|z_{1}|)+v(|z_{2}|)\right) \left( \frac{\omega
(h_{1})}{v(h_{1})}+\frac{\omega (h_{2})}{v(h_{2})}\right) \right) .
\end{align*}%
Substituting $t_{1}$ with $t_{1}+h_{1}$ in
\begin{align*}
&&\frac{4}{mn\pi ^{2}}\int_{h_{1}}^{\pi }\int_{h_{2}}^{\pi
}H_{x,z_{1},y,z_{2}}(t_{1},t_{2})S(t_1,t_2) \left( \frac{1}{\left( 2\sin \frac{t_{2}}{2}\right) ^{2}}-\frac{1}{%
t_{2}{}^{2}}\right) \frac{1}{t_{1}{}^{2}}dt_{1}dt_{2}
\end{align*}%
we obtain
\begin{align*}
J_{42} &= -\frac{4}{mn\pi ^{2}}\int_{0}^{\pi -h_{1}}\int_{h_{2}}^{\pi
}H_{x,z_{1},y,z_{2}}(t_{1}+h_{1},t_{2})S(t_1,t_2) \\
& \qquad \times \left( \frac{1}{\left( 2\sin \frac{t_{2}}{2}\right) ^{2}}-\frac{1}{%
t_{2}{}^{2}}\right) \frac{1}{\left( t_{1}+h_{1}\right) {}^{2}}dt_{1}dt_{2}.
\end{align*}%
Hence
\begin{align*}
J_{42} &= \frac{2}{mn\pi ^{2}}\int_{h_{1}}^{\pi }\int_{h_{2}}^{\pi
}H_{x,z_{1},y,z_{2}}(t_{1},t_{2})S(t_{1},t_{2}) \left( \frac{1}{\left( 2\sin \frac{t_{2}}{2}\right) ^{2}}-\frac{1}{%
t_{2}{}^{2}}\right) \frac{1}{t_{1}^{2}}dt_{1}dt_{2}\\
& \quad -\frac{2}{mn\pi ^{2}}\int_{0}^{\pi -h_{1}}\int_{h_{2}}^{\pi
}H_{x,z_{1},y,z_{2}}(t_{1}+h_{1},t_{2})S(t_{1},t_{2}) \\
& \qquad \times \left( \frac{1}{\left( 2\sin \frac{t_{2}}{2}\right) ^{2}}-\frac{1}{%
t_{2}{}^{2}}\right) \frac{1}{\left( t_{1}+h_{1}\right) {}^{2}}dt_{1}dt_{2}\\
&= \frac{2}{mn\pi ^{2}}\int_{h_{1}}^{\pi }\int_{h_{2}}^{\pi
}H_{x,z_{1},y,z_{2}}(t_{1},t_{2})S(t_{1},t_{2}) \left( \frac{1}{\left( 2\sin \frac{t_{2}}{2}\right) ^{2}}-\frac{1}{%
t_{2}{}^{2}}\right) \frac{1}{t_{1}^{2}}dt_{1}dt_{2}\\
& \quad -\frac{2}{mn\pi ^{2}}\left( \int_{0}^{h_{1}}\int_{h_{2}}^{\pi
}+\int_{h_{1}}^{\pi }\int_{h_{2}}^{\pi }-\int_{\pi -h_{1}}^{\pi
}\int_{h_{2}}^{\pi }\right) H_{x,z_{1},y,z_{2}}(t_{1}+h_{1},t_{2}) \\
& \qquad \times S(t_{1},t_{2})\left( \frac{1}{\left( 2\sin \frac{t_{2}}{2}\right)
^{2}}-\frac{1}{t_{2}{}^{2}}\right) \frac{1}{\left( t_{1}+h_{1}\right) {}^{2}}%
dt_{1}dt_{2}\\
&= \frac{2}{mn\pi ^{2}}\int_{h_{1}}^{\pi }\int_{h_{2}}^{\pi }\left( \frac{%
H_{x,z_{1},y,z_{2}}(t_{1},t_{2})}{t_{1}{}^{2}}-\frac{%
H_{x,z_{1},y,z_{2}}(t_{1}+h_{1},t_{2})}{\left( t_{1}+h_{1}\right) {}^{2}}%
\right) \\
& \qquad \times S(t_{1},t_{2})\left( \frac{1}{\left( 2\sin \frac{t_{2}}{2}\right)
^{2}}-\frac{1}{t_{2}{}^{2}}\right) \frac{1}{t_{1}{}^{2}}dt_{1}dt_{2}\\
& \quad -\frac{2}{mn\pi ^{2}}\left( \int_{0}^{h_{1}}\int_{h_{2}}^{\pi }-\int_{\pi
-h_{1}}^{\pi }\int_{h_{2}}^{\pi }\right)
H_{x,z_{1},y,z_{2}}(t_{1}+h_{1},t_{2}) \\
& \qquad \times S(t_{1},t_{2})\left( \frac{1}{\left( 2\sin \frac{t_{2}}{2}\right)
^{2}}-\frac{1}{t_{2}{}^{2}}\right) \frac{1}{\left( t_{1}+h_{1}\right) {}^{2}}%
dt_{1}dt_{2} \\
&:= \sum\limits_{s=1}^{3}J_{42}^{\left( s\right) }.
\end{align*}%
Since the function 
\begin{equation*}
\frac{1}{\left( 2\sin \frac{t_{2}}{2}\right) ^{2}}-\frac{1}{t_{2}^{2}}
\end{equation*}%
is bounded for $0<t_{1}\leq \pi $, using Lemma \ref{le1} with $p\geq 1$ and
Lemma \ref{le2}, we have%
\begin{align*}
\left\Vert J_{42}^{\left( 2\right) }\right\Vert _{p} &= \mathcal{O}\left( 
\frac{1}{mn}\right) \int_{0}^{h_{1}}\int_{h_{2}}^{\pi }\left(
v(|z_{1}|)+v(|z_{2}|)\right) \left( \frac{\omega (t_{1}+h_{1})}{%
v(t_{1}+h_{1})}+\frac{\omega (t_{2})}{v(t_{2})}\right) \frac{%
m^{2}t_{1}^{2}dt_{1}dt_{2}}{(t_{1}+h_{1})^{2}} \\
&= \mathcal{O}\left( \frac{m}{n}\right) \left( v(|z_{1}|)+v(|z_{2}|)\right)
\left( \frac{\omega (2h_{1})}{v(2h_{1})}+\frac{\omega (\pi )}{v(\pi )}%
\right) h_{1} \\
&= \mathcal{O}\left( \left( v(|z_{1}|)+v(|z_{2}|)\right) \left( \frac{\omega
(h_{1})}{v(h_{1})}+\frac{\omega (h_{2})}{v(h_{2})}\right) \right) ,
\end{align*}%
\begin{align*}
\left\Vert J_{42}^{\left( 3\right) }\right\Vert _{p} &= \mathcal{O}\left( 
\frac{1}{mn}\right) \int_{\pi -h_{1}}^{\pi }\int_{h_{2}}^{\pi }\left(
v(|z_{1}|)+v(|z_{2}|)\right) \left( \frac{\omega (t_{1}+h_{1})}{%
v(t_{1}+h_{1})}+\frac{\omega (t_{2})}{v(t_{2})}\right) \frac{dt_{1}dt_{2}}{%
(t_{1}+h_{1})^{2}} \\
&= \mathcal{O}\left( \frac{1}{mn}\right) \left( v(|z_{1}|)+v(|z_{2}|)\right)
\left( \frac{\omega (\pi )}{v(\pi )}\right) h_{1} \\
&= \mathcal{O}\left( \frac{1}{m}\left( v(|z_{1}|)+v(|z_{2}|)\right) \left( 
\frac{\omega (h_{1})}{v(h_{1})}+\frac{\omega (h_{2})}{v(h_{2})}\right)
\right)
\end{align*}%
and%
\begin{align*}
\left\Vert J_{42}^{\left( 1\right) }\right\Vert _{p} &= \mathcal{O}\left( 
\frac{1}{mn}\right) \int_{h_{1}}^{\pi }\int_{h_{2}}^{\pi }\left\Vert
H_{x,z_{1},y,z_{2}}(t_{1},t_{2})-H_{x,z_{1},y,z_{2}}(t_{1}+h_{1},t_{2})%
\right\Vert _{p}\frac{1}{t_{1}^{2}}dt_{1}dt_{2} \\
& \quad +\mathcal{O}\left( \frac{1}{mn}\right) \int_{h_{1}}^{\pi
}\int_{h_{2}}^{\pi }\left\Vert
H_{x,z_{1},y,z_{2}}(t_{1}+h_{1},t_{2})\right\Vert _{p}\frac{%
h_{1}(2t_{1}+h_{1})}{t_{1}^{2}\left( t_{1}+h_{1}\right) ^{2}}dt_{1}dt_{2} \\
&= \mathcal{O}\left( \frac{1}{mn}\right) \left( v(|z_{1}|)+v(|z_{2}|)\right)
\left( \frac{\omega (h_{1})}{v(h_{1})}+\frac{\omega (h_{2})}{v(h_{2})}%
\right) \int_{h_{1}}^{\pi }\frac{1}{t_{1}^{2}}dt_{1} \\
& \quad +\mathcal{O}\left( \frac{1}{m^{2}n}\right) \int_{h_{1}}^{\pi
}\int_{h_{2}}^{\pi }\left( v(|z_{1}|)+v(|z_{2}|)\right) \\
& \qquad \times \left( \frac{\omega (t_{1}+h_{1})}{\left( t_{1}+h_{1}\right)
v(t_{1}+h_{1})}+\frac{\omega (t_{2})}{t_{1}v(t_{2})}\right) \frac{1}{%
t_{1}^{2}}dt_{1}dt_{2} \\
&= \mathcal{O}\left( \frac{1}{n}\right) \left( v(|z_{1}|)+v(|z_{2}|)\right)
\left( \frac{\omega (h_{1})}{v(h_{1})}+\frac{\omega (h_{2})}{v(h_{2})}\right)
\\
& \quad +\mathcal{O}\left( \frac{1}{m^{2}n}\right) \left(
v(|z_{1}|)+v(|z_{2}|)\right) \left( \frac{\omega (2h_{1})}{2h_{1}v(h_{1})}%
\int_{h_{1}}^{\pi }\frac{1}{t_{1}^{2}}dt_{1}+\frac{\omega (\pi )}{v(\pi )}%
\int_{h_{1}}^{\pi }\frac{1}{t_{1}^{3}}dt_{1}\right) \\
&= \mathcal{O}\left( \left( v(|z_{1}|)+v(|z_{2}|)\right) \left( \frac{\omega
(h_{1})}{v(h_{1})}+\frac{\omega (h_{2})}{v(h_{2})}\right) \right) .
\end{align*}%
The quantity $J_{43}$ can be estimated analogously to $J_{42}.$

Substituting in
\begin{equation*}
J_{44}=\frac{4}{mn\pi ^{2}}\int_{h_{1}}^{\pi }\int_{h_{2}}^{\pi
}H_{x,z_{1},y,z_{2}}(t_{1},t_{2})\frac{S(t_1,t_2)}{\left( t_{1}t_{2}\right)
^{2}}dt_{1}dt_{2}
\end{equation*}%
$t_{1}$ with $t_{1}+h_{1}$, $t_{2}$ with $t_{2}+h_{2}$ and $t_{1}$ with $%
t_{1}+h_{1}$, $t_{2}$ with $t_{2}+h_{2}$, respectively, we get
\begin{equation*}
J_{44}=-\frac{4}{mn\pi ^{2}}\int_{0}^{\pi -h_{1}}\int_{h_{2}}^{\pi
}H_{x,z_{1},y,z_{2}}(t_{1}+h_{1},t_{2})\frac{S(t_1,t_2)}{\left(
(t_{1}+h_{1})t_{2}\right) ^{2}}dt_{1}dt_{2},
\end{equation*}%
\begin{equation*}
J_{44}=-\frac{4}{mn\pi ^{2}}\int_{h_{1}}^{\pi }\int_{0}^{\pi
-h_{2}}H_{x,z_{1},y,z_{2}}(t_{1},t_{2}+h_{2})\frac{S(t_1,t_2)}{\left(
t_{1}(t_{2}+h_{2})\right) ^{2}}dt_{1}dt_{2},
\end{equation*}%
\begin{equation*}
J_{44}=\frac{4}{mn\pi ^{2}}\int_{0}^{\pi -h_{1}}\int_{0}^{\pi
-h_{2}}H_{x,z_{1},y,z_{2}}(t_{1}+h_{1},t_{2}+h_{2})\frac{S(t_1,t_2)}{\left(
(t_{1}+h_{1})(t_{2}+h_{2})\right) ^{2}}dt_{1}dt_{2}.
\end{equation*}

Hence%
\begin{align*}
J_{44}& =\frac{1}{mn\pi ^{2}}\int_{h_{1}}^{\pi }\int_{h_{2}}^{\pi
}H_{x,z_{1},y,z_{2}}(t_{1},t_{2})\frac{S(t_1,t_2)}{\left( t_{1}t_{2}\right)
^{2}}dt_{1}dt_{2}\\
& \quad -\frac{1}{mn\pi ^{2}}\int_{0}^{\pi -h_{1}}\int_{h_{2}}^{\pi
}H_{x,z_{1},y,z_{2}}(t_{1}+h_{1},t_{2})\frac{S(t_1,t_2)}{\left(
(t_{1}+h_{1})t_{2}\right) ^{2}}dt_{1}dt_{2}\\
& \quad -\frac{1}{mn\pi ^{2}}\int_{h_{1}}^{\pi }\int_{0}^{\pi
-h_{2}}H_{x,z_{1},y,z_{2}}(t_{1},t_{2}+h_{2})\frac{S(t_1,t_2)}{\left(
t_{1}(t_{2}+h_{2})\right) ^{2}}dt_{1}dt_{2}\\
& \quad +\frac{1}{mn\pi ^{2}}\int_{0}^{\pi -h_{1}}\int_{0}^{\pi
-h_{2}}H_{x,z_{1},y,z_{2}}(t_{1}+h_{1},t_{2}+h_{2})\frac{S(t_1,t_2)}{\left(
(t_{1}+h_{1})(t_{2}+h_{2})\right) ^{2}}dt_{1}dt_{2}\\
& =\frac{1}{mn\pi ^{2}}\int_{h_{1}}^{\pi }\int_{h_{2}}^{\pi
}H_{x,z_{1},y,z_{2}}(t_{1},t_{2})\frac{S(t_1,t_2)}{\left( t_{1}t_{2}\right)
^{2}}dt_{1}dt_{2}\\
& \quad -\frac{1}{mn\pi ^{2}}\left( \int_{0}^{h_{1}}\int_{h_{2}}^{\pi
}+\int_{h_{1}}^{\pi }\int_{h_{2}}^{\pi }-\int_{\pi -h_{1}}^{\pi
}\int_{h_{2}}^{\pi }\right) \\
& \quad \qquad H_{x,z_{1},y,z_{2}}(t_{1}+h_{1},t_{2})\frac{S(t_1,t_2)}{\left(
(t_{1}+h_{1})t_{2}\right) ^{2}}dt_{1}dt_{2} \\
& \quad -\frac{1}{mn\pi ^{2}}\left( \int_{h_{1}}^{\pi
}\int_{0}^{h_{2}}+\int_{h_{1}}^{\pi }\int_{h_{2}}^{\pi }-\int_{h_{1}}^{\pi
}\int_{\pi -h_{2}}^{\pi }\right) \\
& \quad \qquad H_{x,z_{1},y,z_{2}}(t_{1},t_{2}+h_{2})\frac{S(t_1,t_2)}{\left(
t_{1}(t_{2}+h_{2})\right) ^{2}}dt_{1}dt_{2}\\
& \quad +\frac{1}{mn\pi ^{2}}\left(
\int_{0}^{h_{1}}\int_{0}^{h_{2}}+\int_{0}^{h_{1}}\int_{h_{2}}^{\pi
}-\int_{0}^{h_{1}}\int_{\pi -h_{2}}^{\pi }+\int_{h_{1}}^{\pi
}\int_{0}^{h_{2}}+\int_{h_{1}}^{\pi }\int_{h_{2}}^{\pi }\right.\\
& \quad \qquad \qquad \left. -\int_{h_{1}}^{\pi }\int_{\pi -h_{2}}^{\pi }-\int_{\pi -h_{1}}^{\pi
}\int_{0}^{h_{2}}-\int_{\pi -h_{1}}^{\pi }\int_{h_{2}}^{\pi }+\int_{\pi
-h_{1}}^{\pi }\int_{\pi -h_{2}}^{\pi }\right)\\
& \quad \qquad \qquad \quad H_{x,z_{1},y,z_{2}}(t_{1}+h_{1},t_{2}+h_{2})\frac{S(t_1,t_2)}{\left(
(t_{1}+h_{1})(t_{2}+h_{2})\right) ^{2}}dt_{1}dt_{2}\\
& =\frac{1}{mn\pi ^{2}}\int_{h_{1}}^{\pi }\int_{h_{2}}^{\pi }\left( \frac{%
H_{x,z_{1},y,z_{2}}(t_{1},t_{2})}{\left( t_{1}t_{2}\right) ^{2}}-\frac{%
H_{x,z_{1},y,z_{2}}(t_{1}+h_{1},t_{2})}{\left( (t_{1}+h_{1})t_{2}\right) ^{2}%
}\right.\\
& \quad \qquad \quad  \left. -\frac{H_{x,z_{1},y,z_{2}}(t_{1},t_{2}+h_{2})}{\left(
t_{1}(t_{2}+h_{2})\right) ^{2}}+\frac{%
H_{x,z_{1},y,z_{2}}(t_{1}+h_{1},t_{2}+h_{2})}{\left(
(t_{1}+h_{1})(t_{2}+h_{2})\right) ^{2}}\right) S(t_1,t_2) dt_{1}dt_{2}\\
& \quad -\frac{1}{mn\pi ^{2}}\left( \int_{0}^{h_{1}}\int_{h_{2}}^{\pi }-\int_{\pi
-h_{1}}^{\pi }\int_{h_{2}}^{\pi }\right)
H_{x,z_{1},y,z_{2}}(t_{1}+h_{1},t_{2})\frac{S(t_1,t_2)}{\left(
(t_{1}+h_{1})t_{2}\right) ^{2}}dt_{1}dt_{2}\\
& \quad -\frac{1}{mn\pi ^{2}}\left( \int_{h_{1}}^{\pi
}\int_{0}^{h_{2}}-\int_{h_{1}}^{\pi }\int_{\pi -h_{2}}^{\pi }\right)
H_{x,z_{1},y,z_{2}}(t_{1},t_{2}+h_{2})\frac{S(t_1,t_2)}{\left(
t_{1}(t_{2}+h_{2})\right) ^{2}}dt_{1}dt_{2}\\
& \quad +\frac{1}{mn\pi ^{2}}\left(
\int_{0}^{h_{1}}\int_{0}^{h_{2}}+\int_{0}^{h_{1}}\int_{h_{2}}^{\pi
}-\int_{0}^{h_{1}}\int_{\pi -h_{2}}^{\pi }+\int_{h_{1}}^{\pi
}\int_{0}^{h_{2}}\right.\\
& \quad \qquad \qquad \left. -\int_{h_{1}}^{\pi }\int_{\pi -h_{2}}^{\pi }-\int_{\pi -h_{1}}^{\pi
}\int_{0}^{h_{2}}-\int_{\pi -h_{1}}^{\pi }\int_{h_{2}}^{\pi }+\int_{\pi
-h_{1}}^{\pi }\int_{\pi -h_{2}}^{\pi }\right)\\
& \quad \qquad \qquad \quad H_{x,z_{1},y,z_{2}}(t_{1}+h_{1},t_{2}+h_{2})\frac{S(t_1,t_2)}{\left(
(t_{1}+h_{1})(t_{2}+h_{2})\right) ^{2}}dt_{1}dt_{2} \\
& :=\sum \limits_{l=1}^{13}J_{44}^{\left( l\right) }.
\end{align*}

We start with the estimate of $J_{44}^{\left( 1\right) }$.
Indeed, we have
\begin{align*}
J_{44}^{\left( 1\right) } &= \frac{1}{mn\pi ^{2}}\int_{h_{1}}^{\pi
}\int_{h_{2}}^{\pi }\Bigg(\frac{H_{x,z_{1},y,z_{2}}(t_{1},t_{2})}{\left(
t_{1}t_{2}\right) ^{2}}-\frac{H_{x,z_{1},y,z_{2}}(t_{1}+h_{1},t_{2})}{\left(
(t_{1}+h_{1})t_{2}\right) ^{2}} \\
& \quad \qquad -\frac{H_{x,z_{1},y,z_{2}}(t_{1},t_{2}+h_{2})}{\left(
t_{1}(t_{2}+h_{2})\right) ^{2}}+\frac{%
H_{x,z_{1},y,z_{2}}(t_{1}+h_{1},t_{2}+h_{2})}{\left(
(t_{1}+h_{1})(t_{2}+h_{2})\right) ^{2}}\Bigg)S(t_{1},t_{2})dt_{1}dt_{2} \\
&= \frac{1}{mn\pi ^{2}}\int_{h_{1}}^{\pi }\int_{h_{2}}^{\pi }\Bigg(%
H_{x,z_{1},y,z_{2}}(t_{1},t_{2})-H_{x,z_{1},y,z_{2}}(t_{1}+h_{1},t_{2}) \\
&\quad \qquad -H_{x,z_{1},y,z_{2}}(t_{1},t_{2}+h_{2})+H_{x,z_{1},y,z_{2}}(t_{1}+h_{1},t_{2}+h_{2})%
\Bigg)\frac{S(t_{1},t_{2})dt_{1}dt_{2}}{\left( t_{1}t_{2}\right) ^{2}} \\
&\quad +\frac{1}{mn\pi ^{2}}\int_{h_{1}}^{\pi }\int_{h_{2}}^{\pi }\Bigg(%
H_{x,z_{1},y,z_{2}}(t_{1}+h_{1},t_{2})-H_{x,z_{1},y,z_{2}}(t_{1}+h_{1},t_{2}+h_{2})%
\Bigg) \\
&\qquad \times \Bigg(\frac{1}{\left( t_{1}t_{2}\right) ^{2}}-\frac{1}{\left(
(t_{1}+h_{1})t_{2}\right) ^{2}}\Bigg)S(t_{1},t_{2})dt_{1}dt_{2} \\
&\quad +\frac{1}{mn\pi ^{2}}\int_{h_{1}}^{\pi }\int_{h_{2}}^{\pi }\Bigg(%
H_{x,z_{1},y,z_{2}}(t_{1},t_{2}+h_{2})-H_{x,z_{1},y,z_{2}}(t_{1}+h_{1},t_{2}+h_{2})%
\Bigg) \\
&\qquad \times \Bigg(\frac{1}{\left( t_{1}t_{2}\right) ^{2}}-\frac{1}{\left(
t_{1}(t_{2}+h_{2})\right) ^{2}}\Bigg)S(t_{1},t_{2})dt_{1}dt_{2} \\
&\quad +\frac{1}{mn\pi ^{2}}\int_{h_{1}}^{\pi }\int_{h_{2}}^{\pi
}H_{x,z_{1},y,z_{2}}(t_{1}+h_{1},t_{2}+h_{2})\Bigg(\frac{1}{(t_{1}+h_{1})^{2}%
}-\frac{1}{t_{1}^{2}}\Bigg) \\
&\qquad \times \Bigg(\frac{1}{(t_{2}+h_{2})^{2}}-\frac{1}{t_{2}^{2}}\Bigg)%
S(t_{1},t_{2})dt_{1}dt_{2} \\
&:= J_{44}^{\left( 1,1\right) }+J_{44}^{\left( 1,2\right) }+J_{44}^{\left(
1,3\right) }+J_{44}^{\left( 1,4\right) }.
\end{align*}

It follows from Lemma \ref{le2} (iv) and $|S(t_{1},t_{2})|\leq 1$ for $%
t_{1}\in (h_{1},\pi )$, $t_{2}\in (h_{2},\pi )$ that 
\begin{align*}
\Vert J_{44}^{\left( 1,1\right) }\Vert _{p} &= \frac{\mathcal{O}(1)}{mn}%
\int_{h_{1}}^{\pi }\int_{h_{2}}^{\pi }\Big\|%
H_{x,z_{1},y,z_{2}}(t_{1},t_{2})-H_{x,z_{1},y,z_{2}}(t_{1}+h_{1},t_{2}) \\
& \quad -H_{x,z_{1},y,z_{2}}(t_{1},t_{2}+h_{2})+H_{x,z_{1},y,z_{2}}(t_{1}+h_{1},t_{2}+h_{2})%
\Big\|_{p}\frac{|S(t_{1},t_{2})|dt_{1}dt_{2}}{\left( t_{1}t_{2}\right) ^{2}}
\\
&= \frac{\mathcal{O}(1)}{mn}(v(|z_{1}|)+(v(|z_{2}|))\left( \frac{\omega
(h_{1})}{v(h_{1})}+\frac{\omega (h_{2})}{v(h_{2})}\right) \int_{h_{1}}^{\pi
}\int_{h_{2}}^{\pi }\frac{dt_{1}dt_{2}}{\left( t_{1}t_{2}\right) ^{2}} \\
&= \mathcal{O}(1)(v(|z_{1}|)+(v(|z_{2}|))\left( \frac{\omega (h_{1})}{%
v(h_{1})}+\frac{\omega (h_{2})}{v(h_{2})}\right) .
\end{align*}
Moreover, from Lemma \ref{le2} (iii) and using the same arguments as above,
we obtain 
\begin{align*}
\Vert J_{44}^{\left( 1,2\right) }\Vert _{p} &= \frac{\mathcal{O}(1)}{mn}%
\int_{h_{1}}^{\pi }\int_{h_{2}}^{\pi }\Big\|%
H_{x,z_{1},y,z_{2}}(t_{1}+h_{1},t_{2})-H_{x,z_{1},y,z_{2}}(t_{1}+h_{1},t_{2}+h_{2})%
\Big\|_{p} \\
& \quad \times \frac{1}{t_{2}^{2}}\Bigg(\frac{1}{t_{1}^{2}}-\frac{1}{\left(
t_{1}+h_{1}\right) ^{2}}\Bigg)|S(t_{1},t_{2})|dt_{1}dt_{2} \\
&= \frac{\mathcal{O}(1)}{mn}(v(|z_{1}|)+(v(|z_{2}|))\left( \frac{\omega
(h_{1})}{v(h_{1})}+\frac{\omega (h_{2})}{v(h_{2})}\right) \int_{h_{1}}^{\pi
}\int_{h_{2}}^{\pi }\frac{h_{1}\left( 2t_{1}+h_{1}\right) }{%
t_{2}^{2}t_{1}^{2}\left( t_{1}+h_{1}\right) ^{2}}dt_{1}dt_{2} \\
&= \frac{\mathcal{O}(1)}{m^{2}n}\left( (v(|z_{1}|)+(v(|z_{2}|)\right) \left( 
\frac{\omega (h_{1})}{v(h_{1})}+\frac{\omega (h_{2})}{v(h_{2})}\right)
\int_{h_{1}}^{\pi }\int_{h_{2}}^{\pi }\frac{1}{t_{2}^{2}t_{1}^{3}}%
dt_{1}dt_{2} \\
&= \mathcal{O}(1)(v(|z_{1}|)+(v(|z_{2}|))\left( \frac{\omega (h_{1})}{%
v(h_{1})}+\frac{\omega (h_{2})}{v(h_{2})}\right) .
\end{align*}
Similarly, we obtain 
\begin{equation*}
\Vert J_{44}^{\left( 1,3\right) }\Vert _{p}=\mathcal{O}%
(1)(v(|z_{1}|)+(v(|z_{2}|))\left( \frac{\omega (h_{1})}{v(h_{1})}+\frac{%
\omega (h_{2})}{v(h_{2})}\right) .
\end{equation*}
To estimate $\Vert J_{44}^{\left( 1,4\right) }\Vert _{p}$ we use Lemma \ref%
{le1} and Lemma \ref{le2} (iii) in order to get 
\begin{align*}
\Vert J_{44}^{\left( 1,4\right) }\Vert _{p} &= \frac{\mathcal{O}(1)}{mn}%
\int_{h_{1}}^{\pi }\int_{h_{2}}^{\pi }\Big\|%
H_{x,z_{1},y,z_{2}}(t_{1}+h_{1},t_{2}+h_{2})\Big\|_{p} \\
& \quad \times \Bigg(\frac{1}{(t_{1}+h_{1})^{2}}-\frac{1}{t_{1}^{2}}\Bigg)\Bigg(%
\frac{1}{(t_{2}+h_{2})^{2}}-\frac{1}{t_{2}^{2}}\Bigg)%
|S(t_{1},t_{2})|dt_{1}dt_{2} \\
&= \frac{\mathcal{O}(1)}{mn}\left( (v(|z_{1}|)+(v(|z_{2}|)\right)
\int_{h_{1}}^{\pi }\int_{h_{2}}^{\pi }\Bigg(\frac{\omega (t_{1}+h_{1})}{%
v(t_{1}+h_{1})}+\frac{\omega (t_{2}+h_{2})}{v(t_{2}+h_{2})}\Bigg) \\
& \quad \times \frac{h_{1}h_{2}\left( 2t_{1}+h_{1}\right) \left(
2t_{2}+h_{2}\right) }{\left( t_{1}t_{2}\left( t_{1}+h_{1}\right) \left(
t_{2}+h_{2}\right) \right) ^{2}}dt_{1}dt_{2} \\
&= \frac{\mathcal{O}(1)}{\left( mn\right) ^{2}}\left(
(v(|z_{1}|)+(v(|z_{2}|)\right) \\
& \quad \times \int_{h_{1}}^{\pi }\int_{h_{2}}^{\pi }\Bigg(\frac{\omega
(t_{1}+h_{1})}{\left( t_{1}+h_{1}\right) v(h_{1})t_{2}^{3}t_{1}^{2}}+\frac{%
\omega (t_{2}+h_{2})}{\left( t_{2}+h_{2}\right) v(h_{2})t_{1}^{3}t_{2}^{2}}%
\Bigg)dt_{1}dt_{2} \\
&= \frac{\mathcal{O}(1)}{\left( mn\right) ^{2}}\left(
(v(|z_{1}|)+(v(|z_{2}|)\right) \\
& \quad \times \Bigg(\frac{\omega (2h_{1})}{2h_{1}v(h_{1})}\int_{h_{1}}^{\pi
}\int_{h_{2}}^{\pi }\frac{1}{t_{2}^{3}t_{1}^{2}}dt_{1}dt_{2}+\frac{\omega
(2h_{2})}{2h_{2}v(h_{2})}\int_{h_{1}}^{\pi }\int_{h_{2}}^{\pi }\frac{1}{%
t_{2}^{2}t_{1}^{3}}dt_{1}dt_{2}\Bigg) \\
&= \mathcal{O}(1)(v(|z_{1}|)+(v(|z_{2}|))\left( \frac{\omega (h_{1})}{%
v(h_{1})}+\frac{\omega (h_{2})}{v(h_{2})}\right) .
\end{align*}
Whence, using the above estimates, we have 
\begin{equation}
\Vert J_{44}^{\left( 1\right) }\Vert _{p}=\mathcal{O}%
(1)(v(|z_{1}|)+(v(|z_{2}|))\left( \frac{\omega (h_{1})}{v(h_{1})}+\frac{%
\omega (h_{2})}{v(h_{2})}\right) .  \label{eqb1}
\end{equation}

Replacing $t_{2}$ with $t_{2}+h_{2}$ in 
\begin{equation*}
J_{44}^{\left( 2\right) }=\frac{1}{mn\pi ^{2}}\int_{0}^{h_{1}}\int_{h_{2}}^{%
\pi }H_{x,z_{1},y,z_{2}}(t_{1}+h_{1},t_{2})\frac{S(t_{1},t_{2})}{\left(
(t_{1}+h_{1})t_{2}\right) ^{2}}dt_{1}dt_{2}
\end{equation*}%
we get 
\begin{equation*}
J_{44}^{\left( 2\right) }=-\frac{1}{mn\pi ^{2}}\int_{0}^{h_{1}}\int_{0}^{\pi
-h_{2}}H_{x,z_{1},y,z_{2}}(t_{1}+h_{1},t_{2}+h_{2})\frac{S(t_{1},t_{2})}{%
\left( (t_{1}+h_{1})(t_{2}+h_{2})\right) ^{2}}dt_{1}dt_{2}
\end{equation*}%
and after adding them we obtain 
\begin{align*}
J_{44}^{\left( 2\right) } &= \frac{1}{2mn\pi ^{2}}\bigg(\int_{0}^{h_{1}}%
\int_{h_{2}}^{\pi }H_{x,z_{1},y,z_{2}}(t_{1}+h_{1},t_{2})\frac{S(t_{1},t_{2})%
}{\left( (t_{1}+h_{1})t_{2}\right) ^{2}}dt_{1}dt_{2} \\
& \quad -\int_{0}^{h_{1}}\int_{0}^{\pi
-h_{2}}H_{x,z_{1},y,z_{2}}(t_{1}+h_{1},t_{2}+h_{2})\frac{S(t_{1},t_{2})}{%
\left( (t_{1}+h_{1})(t_{2}+h_{2})\right) ^{2}}dt_{1}dt_{2}\bigg) \\
&= \frac{1}{2mn\pi ^{2}}\bigg(\int_{0}^{h_{1}}\int_{h_{2}}^{\pi }\bigg(\frac{%
H_{x,z_{1},y,z_{2}}(t_{1}+h_{1},t_{2})}{\left( (t_{1}+h_{1})t_{2}\right) ^{2}%
}-\frac{H_{x,z_{1},y,z_{2}}(t_{1}+h_{1},t_{2}+h_{2})}{\left(
(t_{1}+h_{1})(t_{2}+h_{2})\right) ^{2}}\bigg) \\
&  \qquad \times
S(t_{1},t_{2})dt_{1}dt_{2}\\
& \quad -\int_{0}^{h_{1}} \int_{0}^{h_{2}}H_{x,z_{1},y,z_{2}}(t_{1}+h_{1},t_{2}+h_{2})\frac{%
S(t_{1},t_{2})}{\left( (t_{1}+h_{1})(t_{2}+h_{2})\right) ^{2}}dt_{1}dt_{2} \\
& \quad +\int_{0}^{h_{1}}\int_{\pi -h_{2}}^{\pi
}H_{x,z_{1},y,z_{2}}(t_{1}+h_{1},t_{2}+h_{2})\frac{S(t_{1},t_{2})}{\left(
(t_{1}+h_{1})(t_{2}+h_{2})\right) ^{2}}dt_{1}dt_{2}\bigg) \\
&:= \sum\limits_{s=1}^{3}J_{44s}^{\left( 2\right) }.
\end{align*}

Since 
\begin{equation*}
|\sin 2mt_{1}\sin mt_{1}\sin 2nt_{2}\sin nt_{2}|\leq
4(mnt_{1}t_{2})^{2}\quad \text{for}\quad 0<t_{1}<\pi ,0<t_{2}<\pi ,
\end{equation*}%
then applying Lemma \ref{le1} and Lemma \ref{le2} (iii) we have 
\begin{align*}
\Vert J_{442}^{\left( 2\right) }\Vert _{p} &= \mathcal{O}\left(\frac{1}{mn}%
\right)\int_{0}^{h_{1}}\int_{0}^{h_{2}}\left\Vert
H_{x,z_{1},y,z_{2}}(t_{1}+h_{1},t_{2}+h_{2})\right\Vert _{p}\frac{%
(mnt_{1}t_{2})^{2}}{\left( (t_{1}+h_{1})(t_{2}+h_{2})\right) ^{2}}%
dt_{1}dt_{2} \\
&= \mathcal{O}(mn)\int_{0}^{h_{1}}\int_{0}^{h_{2}}\left(
v(|z_{1}|)+v(|z_{2}|)\right) \left( \frac{\omega (t_{1}+h_{1})}{%
v(t_{1}+h_{1})}+\frac{\omega (t_{2}+h_{2})}{v(t_{2}+h_{2})}\right)
dt_{1}dt_{2} \\
&= \mathcal{O}(mn)\left( v(|z_{1}|)+v(|z_{2}|)\right) \left( \frac{\omega
(2h_{1})}{v(2h_{1})}+\frac{\omega (2h_{2})}{v(2h_{2})}\right) h_{1}h_{2} \\
&= \mathcal{O}(1)\left( v(|z_{1}|)+v(|z_{2}|)\right) \left( \frac{\omega
(h_{1})}{v(h_{1})}+\frac{\omega (h_{2})}{v(h_{2})}\right) .
\end{align*}

Because of 
\begin{equation*}
|\sin 2mt_{1}\sin mt_{1}|\leq 2(mt_{1})^{2}\quad \text{for}\quad 0<t_{1}<\pi
\end{equation*}%
and 
\begin{equation*}
|\sin 2nt_{2}\sin nt_{2}|\leq 1\quad \text{for}\quad \pi -h_{2}<t_{2}<\pi ,
\end{equation*}%
it follows that 
\begin{align*}
\Vert J_{443}^{\left( 2\right) }\Vert _{p} &= \mathcal{O}\left(\frac{1}{mn}%
\right)\int_{0}^{h_{1}}\int_{\pi -h_{2}}^{\pi }\left\Vert
H_{x,z_{1},y,z_{2}}(t_{1}+h_{1},t_{2}+h_{2})\right\Vert _{p}\frac{%
(mt_{1})^{2}}{\left( (t_{1}+h_{1})(t_{2}+h_{2})\right) ^{2}}dt_{1}dt_{2} \\
&= \mathcal{O}\left(\frac{m}{n}\right)\int_{0}^{h_{1}}\int_{\pi -h_{2}}^{\pi }\left(
v(|z_{1}|)+v(|z_{2}|)\right) \left( \frac{\omega (t_{1}+h_{1})}{%
v(t_{1}+h_{1})}+\frac{\omega (t_{2}+h_{2})}{v(t_{2}+h_{2})}\right) \frac{%
dt_{1}dt_{2}}{(t_{2}+h_{2})^{2}} \\
&= \mathcal{O}\left(\frac{m}{n}\right)\left( v(|z_{1}|)+v(|z_{2}|)\right) \left( \frac{%
\omega (2h_{1})}{v(2h_{1})}h_{1}h_{2}+h_{1}\int_{\pi }^{\pi +h_{2}}\frac{%
\omega (\theta _{2})}{v(\theta _{2})}\frac{d\theta _{2}}{\theta _{2}^2}\right)
\\
&= \mathcal{O}\left(\frac{m}{n}\right)\left( v(|z_{1}|)+v(|z_{2}|)\right) \left( \frac{%
\omega (h_{1})}{v(h_{1})}h_{1}h_{2}+\frac{\omega (\pi +h_{2})}{v(\pi +h_{2})}%
\frac{h_{1}h_{2}}{\pi (\pi +h_{2})}\right) \\
&= \mathcal{O}\left( \frac{1}{n^{2}}\right) \left(
v(|z_{1}|)+v(|z_{2}|)\right) \left( \frac{\omega (h_{1})}{v(h_{1})}+\frac{%
\omega (\pi )}{v(\pi )}\right) \\
&= \mathcal{O}\left(\frac{1}{n}\right)\left( v(|z_{1}|)+v(|z_{2}|)\right) \left( \frac{%
\omega (h_{1})}{v(h_{1})}+\frac{\omega (h_{2})}{v(h_{2})}\right) .
\end{align*}

Now we write 
\begin{align*}
J_{441}^{\left( 2\right) } &= \frac{1}{2mn\pi ^{2}}\int_{0}^{h_{1}}%
\int_{h_{2}}^{\pi }\bigg(\frac{H_{x,z_{1},y,z_{2}}(t_{1}+h_{1},t_{2})}{%
\left( (t_{1}+h_{1})t_{2}\right) ^{2}}-\frac{%
H_{x,z_{1},y,z_{2}}(t_{1}+h_{1},t_{2}+h_{2})}{\left(
(t_{1}+h_{1})t_{2}\right) ^{2}} \\
& \quad +\frac{H_{x,z_{1},y,z_{2}}(t_{1}+h_{1},t_{2}+h_{2})}{\left(
(t_{1}+h_{1})t_{2}\right) ^{2}}-\frac{%
H_{x,z_{1},y,z_{2}}(t_{1}+h_{1},t_{2}+h_{2})}{\left(
(t_{1}+h_{1})(t_{2}+h_{2})\right) ^{2}}\bigg)S(t_{1},t_{2})dt_{1}dt_{2} \\
&= \frac{1}{2mn\pi ^{2}}\int_{0}^{h_{1}}\int_{h_{2}}^{\pi }\bigg(%
H_{x,z_{1},y,z_{2}}(t_{1}+h_{1},t_{2})-H_{x,z_{1},y,z_{2}}(t_{1}+h_{1},t_{2}+h_{2})%
\bigg) \\
&\qquad \times \frac{S(t_{1},t_{2})}{\left( (t_{1}+h_{1})t_{2}\right) ^{2}}%
dt_{1}dt_{2} \\
&\quad +\int_{0}^{h_{1}}\int_{h_{2}}^{\pi
}H_{x,z_{1},y,z_{2}}(t_{1}+h_{1},t_{2}+h_{2})\frac{S(t_{1},t_{2})}{\left(
t_{1}+h_{1}\right) ^{2}} \bigg(\frac{1}{t_{2}^{2}}-\frac{1}{\left( t_{2}+h_{2}\right) ^{2}}%
\bigg)dt_{1}dt_{2} \\
& :=J_{441}^{\left( 21\right) }+J_{441}^{\left( 22\right) }.
\end{align*}
For $\Vert J_{441}^{\left( 21\right) }\Vert _{p}$ we have 
\begin{align*}
\Vert J_{441}^{\left( 21\right) }\Vert _{p} &= \mathcal{O}\left(\frac{1}{mn}%
\right)\int_{0}^{h_{1}}\int_{h_{2}}^{\pi }\left\Vert
H_{x,z_{1},y,z_{2}}(t_{1}+h_{1},t_{2})-H_{x,z_{1},y,z_{2}}(t_{1}+h_{1},t_{2}+h_{2})\right\Vert _{p}
\\
& \quad \times \frac{(mt_{1})^{2}}{\left( (t_{1}+h_{1})t_{2}\right) ^{2}}%
dt_{1}dt_{2} \\
&= \mathcal{O}\left(\frac{m}{n}\right)\int_{0}^{h_{1}}\int_{h_{2}}^{\pi
}(v(|z_{1}|)+(v(|z_{2}|))\left( \frac{\omega (h_{1})}{v(h_{1})}+\frac{\omega
(h_{2})}{v(h_{2})}\right) \frac{dt_{1}dt_{2}}{t_{2}^{2}} \\
&= \mathcal{O}(1)(v(|z_{1}|)+(v(|z_{2}|))\left( \frac{\omega (h_{1})}{%
v(h_{1})}+\frac{\omega (h_{2})}{v(h_{2})}\right) ,
\end{align*}%
while for $\Vert J_{441}^{\left( 22\right) }\Vert _{p}$, we get
\begin{align*}
\Vert J_{441}^{\left( 22\right) }\Vert _{p} &= \mathcal{O}\left(\frac{1}{mn}%
\right)\int_{0}^{h_{1}}\int_{h_{2}}^{\pi }\left\Vert
H_{x,z_{1},y,z_{2}}(t_{1}+h_{1},t_{2}+h_{2})\right\Vert _{p} \\
& \quad \times \frac{(mt_{1})^{2}}{\left( t_{1}+h_{1}\right) ^{2}}\bigg(\frac{1}{%
t_{2}^{2}}-\frac{1}{\left( t_{2}+h_{2}\right) ^{2}}\bigg)dt_{1}dt_{2} \\
&= \mathcal{O}\left(\frac{m}{n}\right)\int_{0}^{h_{1}}\int_{h_{2}}^{\pi
}(v(|z_{1}|)+(v(|z_{2}|)) \\
& \quad \times \left( \frac{\omega (t_{1}+h_{1})}{v(t_{1}+h_{1})}+\frac{\omega
(t_{2}+h_{2})}{v(t_{2}+h_{2})}\right) \frac{h_{2}(2t_{2}+h_{2})}{%
t_{2}^{2}\left( t_{2}+h_{2}\right) ^{2}}dt_{1}dt_{2} \\
&= \mathcal{O}\left(\frac{m}{n^{2}}\right)(v(|z_{1}|)+(v(|z_{2}|))\bigg(%
\int_{0}^{h_{1}}\int_{h_{2}}^{\pi }\frac{\omega (t_{1}+h_{1})}{v(t_{1}+h_{1})%
}\frac{1}{t_{2}^{3}}dt_{1}dt_{2} \\
& \quad +\int_{0}^{h_{1}}\int_{h_{2}}^{\pi }\frac{\omega (t_{2}+h_{2})}{\left(
t_{2}+h_{2}\right) v(h_{2})}\frac{1}{t_{2}^{2}}dt_{1}dt_{2}\bigg) \\
&= \mathcal{O}\left(\frac{m}{n^{2}}\right)(v(|z_{1}|)+(v(|z_{2}|))\bigg(\frac{\omega
(2h_{1})}{v(h_{1})}\int_{h_{2}}^{\pi }\frac{1}{t_{2}^{3}}dt_{2} 
 +\frac{\omega (2h_{2})}{2h_{2}v(h_{2})}\int_{h_{2}}^{\pi }\frac{1}{%
t_{2}^{2}}dt_{2}\bigg)h_{1} \\
&= \mathcal{O}(1)(v(|z_{1}|)+(v(|z_{2}|))\left( \frac{\omega (h_{1})}{%
v(h_{1})}+\frac{\omega (h_{2})}{v(h_{2})}\right) .
\end{align*}
Whence, 
\begin{equation*}
\Vert J_{441}^{\left( 2\right) }\Vert _{p}=\mathcal{O}%
(1)(v(|z_{1}|)+(v(|z_{2}|))\left( \frac{\omega (h_{1})}{v(h_{1})}+\frac{%
\omega (h_{2})}{v(h_{2})}\right) .
\end{equation*}
Thus, we obtain 
\begin{equation}
\Vert J_{44}^{\left( 2\right) }\Vert _{p}=\mathcal{O}(1)\left(
v(|z_{1}|)+v(|z_{2}|)\right) \left( \frac{\omega (h_{1})}{v(h_{1})}+\frac{%
\omega (h_{2})}{v(h_{2})}\right) . \label{eqb2}
\end{equation}

By analogy, we find that 
\begin{equation}
\Vert J_{44}^{\left( 4\right) }\Vert _{p}=\mathcal{O}(1)\left(
v(|z_{1}|)+v(|z_{2}|)\right) \left( \frac{\omega (h_{1})}{v(h_{1})}+\frac{%
\omega (h_{2})}{v(h_{2})}\right) .  \label{eqb3}
\end{equation}

Once more, replacing $t_{2}$ with $t_{2}+h_{2}$ in 
\begin{equation*}
J_{44}^{\left( 3\right) }=\frac{1}{mn\pi ^{2}}\int_{\pi -h_{1}}^{\pi
}\!\!\int_{h_{2}}^{\pi }H_{x,z_{1},y,z_{2}}(t_{1}+h_{1},t_{2})\frac{%
S(t_{1},t_{2})}{\left( (t_{1}+h_{1})t_{2}\right) ^{2}}dt_{1}dt_{2}
\end{equation*}%
we get 
\begin{equation*}
J_{44}^{\left( 3\right) }=-\frac{1}{mn\pi ^{2}}\int_{\pi -h_{1}}^{\pi
}\!\!\int_{0}^{\pi -h_{2}}H_{x,z_{1},y,z_{2}}(t_{1}+h_{1},t_{2}+h_{2})\frac{%
S(t_{1},t_{2})}{\left( (t_{1}+h_{1})(t_{2}+h_{2})\right) ^{2}}dt_{1}dt_{2}
\end{equation*}%
and adding them side by side, we have 
\begin{align*}
J_{44}^{\left( 3\right) } &= \frac{1}{2mn\pi ^{2}}\bigg(\int_{\pi
-h_{1}}^{\pi }\!\!\int_{h_{2}}^{\pi }H_{x,z_{1},y,z_{2}}(t_{1}+h_{1},t_{2})%
\frac{S(t_{1},t_{2})}{\left( (t_{1}+h_{1})t_{2}\right) ^{2}}dt_{1}dt_{2} \\
& \quad -\int_{\pi -h_{1}}^{\pi }\!\!\int_{0}^{\pi
-h_{2}}H_{x,z_{1},y,z_{2}}(t_{1}+h_{1},t_{2}+h_{2})\frac{S(t_{1},t_{2})}{%
\left( (t_{1}+h_{1})(t_{2}+h_{2})\right) ^{2}}dt_{1}dt_{2}\bigg) \\
&= \frac{1}{2mn\pi ^{2}}\bigg(\int_{\pi -h_{1}}^{\pi }\!\!\int_{h_{2}}^{\pi }%
\bigg(\frac{H_{x,z_{1},y,z_{2}}(t_{1}+h_{1},t_{2})}{\left(
(t_{1}+h_{1})t_{2}\right) ^{2}}-\frac{%
H_{x,z_{1},y,z_{2}}(t_{1}+h_{1},t_{2}+h_{2})}{\left(
(t_{1}+h_{1})(t_{2}+h_{2})\right) ^{2}}\bigg) \\
&\quad \times S(t_{1},t_{2})dt_{1}dt_{2} \!- \!\int_{\pi -h_{1}}^{\pi
}\!\!\int_{0}^{h_{2}}H_{x,z_{1},y,z_{2}}(t_{1}+h_{1},t_{2}+h_{2})\frac{%
S(t_{1},t_{2})}{\left( (t_{1}+h_{1})(t_{2}+h_{2})\right) ^{2}}dt_{1}dt_{2} \\
&\quad +\int_{\pi -h_{1}}^{\pi }\!\!\int_{\pi -h_{2}}^{\pi
}H_{x,z_{1},y,z_{2}}(t_{1}+h_{1},t_{2}+h_{2})\frac{S(t_{1},t_{2})}{\left(
(t_{1}+h_{1})(t_{2}+h_{2})\right) ^{2}}dt_{1}dt_{2}\bigg) \\
&:= \sum\limits_{s=1}^{3}J_{44s}^{\left( 3\right) }
\end{align*}

Since 
\begin{equation*}
|\sin 2mt_{1}\sin mt_{1}|\leq 1,\quad \text{for}\quad \pi -h_{1}<t_{1}<\pi
\end{equation*}%
and 
\begin{equation*}
|\sin 2nt_{2}\sin nt_{2}|\leq 2(nt_{2})^{2}\quad \text{for}\quad 0<t_{2}<\pi
,
\end{equation*}%
then applying Lemma \ref{le1} and Lemma \ref{le2} (iii), we have 
\begin{align*}
\Vert J_{442}^{\left( 3\right) }\Vert _{p} &= \mathcal{O}\left(\frac{1}{mn}%
\right)\int_{\pi -h_{1}}^{\pi }\!\!\int_{0}^{h_{2}}\left\Vert
H_{x,z_{1},y,z_{2}}(t_{1}+h_{1},t_{2}+h_{2})\right\Vert _{p}\frac{%
(nt_{2})^{2}}{\left( (t_{1}+h_{1})(t_{2}+h_{2})\right) ^{2}}dt_{1}dt_{2} \\
&= \mathcal{O}\left(\frac{n}{m}\right)\int_{\pi -h_{1}}^{\pi
}\!\!\int_{0}^{h_{2}}\left( v(|z_{1}|)+v(|z_{2}|)\right)
 \left( \frac{\omega (t_{1}+h_{1})}{v(t_{1}+h_{1})}+\frac{\omega
(t_{2}+h_{2})}{v(t_{2}+h_{2})}\right) \frac{dt_{1}dt_{2}}{\left(
t_{1}+h_{1}\right) ^{2}} \\
&= \mathcal{O}\left(\frac{n}{m}\right)\left( v(|z_{1}|)+v(|z_{2}|)\right) \\
& \quad \times h_{2}\left( \int_{\pi -h_{1}}^{\pi }\frac{\omega (t_{1}+h_{1})}{%
v(t_{1}+h_{1})}\frac{dt_{1}}{\left( t_{1}+h_{1}\right) ^{2}}+\frac{\omega
(2h_{2})}{v(2h_{2})}\int_{\pi -h_{1}}^{\pi }\frac{dt_{1}}{\left(
t_{1}+h_{1}\right) ^{2}}\right) \\
&= \mathcal{O}\left(\frac{1}{m}\right)\left( v(|z_{1}|)+v(|z_{2}|)\right) \left(
\int_{\pi }^{\pi +h_{1}}\frac{\omega (\theta _{1})}{v(\theta _{1})}\frac{%
d\theta _{1}}{\theta _{1}^{2}}+\frac{\omega (2h_{2})}{v(2h_{2})}\int_{\pi
}^{\pi +h_{1}}\frac{d\theta _{1}}{\theta _{1}^{2}}\right) \\
&= \mathcal{O}\left(\frac{1}{m}\right)\left( v(|z_{1}|)+v(|z_{2}|)\right) \left( \frac{%
\omega (\pi +h_{1})}{v(\pi +h_{1})}+\frac{\omega (2h_{2})}{v(2h_{2})}\right)
h_{1} \\
&= \mathcal{O}\left( \frac{1}{m^{2}}\right) \left(
v(|z_{1}|)+v(|z_{2}|)\right) \left( \frac{\omega (\pi )}{v(\pi )}+\frac{%
\omega (h_{2})}{v(h_{2})}\right) \\
&= \mathcal{O}\left(\frac{1}{m}\right)\left( v(|z_{1}|)+v(|z_{2}|)\right) \left( \frac{%
\omega (h_{1})}{v(h_{1})}+\frac{\omega (h_{2})}{v(h_{2})}\right) .
\end{align*}
Furthermore, the inequality 
\begin{equation*}
|\sin 2mt_{1}\sin mt_{1}\sin 2nt_{2}\sin nt_{2}|\leq 1\quad \text{for}\quad
\pi -h_{1}<t_{1}<\pi ,\quad \pi -h_{2}<t_{2}<\pi ,
\end{equation*}%
implies 
\begin{align*}
\Vert J_{443}^{\left( 3\right) }\Vert _{p} &= \mathcal{O}\left(\frac{1}{mn}%
\right)\int_{\pi -h_{1}}^{\pi }\!\!\int_{\pi -h_{2}}^{\pi }\left\Vert
H_{x,z_{1},y,z_{2}}(t_{1}+h_{1},t_{2}+h_{2})\right\Vert _{p}\frac{%
dt_{1}dt_{2}}{\left( (t_{1}+h_{1})(t_{2}+h_{2})\right) ^{2}} \\
&= \mathcal{O}\left(\frac{1}{mn}\right)\int_{\pi -h_{1}}^{\pi }\!\!\int_{\pi
-h_{2}}^{\pi }\left( v(|z_{1}|)+v(|z_{2}|)\right) \\
& \quad \times \left( \frac{\omega (t_{1}+h_{1})}{v(t_{1}+h_{1})}+\frac{\omega
(t_{2}+h_{2})}{v(t_{2}+h_{2})}\right) \frac{dt_{1}dt_{2}}{\left(
(t_{1}+h_{1})(t_{2}+h_{2})\right) ^{2}} \\
&= \mathcal{O}\left(\frac{1}{mn}\right)\left( v(|z_{1}|)+v(|z_{2}|)\right) \left( \frac{%
\omega (\pi +h_{1})}{v(\pi +h_{1})}+\frac{\omega (\pi +h_{2})}{v(\pi +h_{2})}%
\right) \\
& \quad \times \int_{\pi -h_{1}}^{\pi }\!\!\int_{\pi -h_{2}}^{\pi }\frac{%
dt_{1}dt_{2}}{\left( (t_{1}+h_{1})(t_{2}+h_{2})\right) ^{2}} \\
&= \mathcal{O}\left(\frac{1}{mn}\right)\left( v(|z_{1}|)+v(|z_{2}|)\right) \frac{\omega
(2\pi )}{v(2\pi )}\int_{\pi }^{\pi +h_{1}}\!\!\int_{\pi }^{\pi +h_{2}}\frac{%
dt_{1}dt_{2}}{\left( t_{1}t_{2}\right) ^{2}} \\
&= \mathcal{O}\left( \frac{1}{\left( mn\right) ^{2}}\right) \left(
v(|z_{1}|)+v(|z_{2}|)\right) \frac{\omega (\pi )}{v(\pi )} \\
&= \mathcal{O}\left(\frac{1}{mn}\right)\left( v(|z_{1}|)+v(|z_{2}|)\right) \left( \frac{%
\omega (h_{1})}{v(h_{1})}+\frac{\omega (h_{2})}{v(h_{2})}\right) .
\end{align*}

It is obvious that 
\begin{align*}
J_{441}^{\left( 3\right) } &= \frac{1}{2mn\pi ^{2}}\int_{\pi -h_{1}}^{\pi
}\!\!\int_{h_{2}}^{\pi }\bigg(\frac{H_{x,z_{1},y,z_{2}}(t_{1}+h_{1},t_{2})}{%
\left( (t_{1}+h_{1})t_{2}\right) ^{2}}-\frac{%
H_{x,z_{1},y,z_{2}}(t_{1}+h_{1},t_{2}+h_{2})}{\left(
(t_{1}+h_{1})t_{2}\right) ^{2}} \\
& \quad +\frac{H_{x,z_{1},y,z_{2}}(t_{1}+h_{1},t_{2}+h_{2})}{\left(
(t_{1}+h_{1})t_{2}\right) ^{2}}-\frac{%
H_{x,z_{1},y,z_{2}}(t_{1}+h_{1},t_{2}+h_{2})}{\left(
(t_{1}+h_{1})(t_{2}+h_{2})\right) ^{2}}\bigg)S(t_{1},t_{2})dt_{1}dt_{2} \\
&= \int_{\pi -h_{1}}^{\pi }\!\!\int_{h_{2}}^{\pi }\bigg(%
H_{x,z_{1},y,z_{2}}(t_{1}+h_{1},t_{2})-H_{x,z_{1},y,z_{2}}(t_{1}+h_{1},t_{2}+h_{2})%
\bigg) \\
& \quad \times \frac{S(t_{1},t_{2})}{\left( (t_{1}+h_{1})t_{2}\right) ^{2}}%
dt_{1}dt_{2}+\int_{\pi -h_{1}}^{\pi }\!\!\int_{h_{2}}^{\pi
}H_{x,z_{1},y,z_{2}}(t_{1}+h_{1},t_{2}+h_{2}) \\
& \quad \times \frac{S(t_{1},t_{2})}{\left( t_{1}+h_{1}\right) ^{2}}\bigg(\frac{1}{%
t_{2}^{2}}-\frac{1}{\left( t_{2}+h_{2}\right) ^{2}}\bigg)%
dt_{1}dt_{2}:=J_{441}^{\left( 31\right) }+J_{441}^{\left( 32\right) }.
\end{align*}
For $\Vert J_{441}^{\left( 31\right) }\Vert _{p}$, we have 
\begin{align*}
\Vert J_{441}^{\left( 31\right) }\Vert _{p} &= \mathcal{O}\left(\frac{1}{mn}%
\right)\int_{\pi -h_{1}}^{\pi }\!\!\int_{h_{2}}^{\pi }\left\Vert
H_{x,z_{1},y,z_{2}}(t_{1}+h_{1},t_{2})-H_{x,z_{1},y,z_{2}}(t_{1}+h_{1},t_{2}+h_{2})\right\Vert _{p}
\\
& \quad \times \frac{dt_{1}dt_{2}}{\left( (t_{1}+h_{1})t_{2}\right) ^{2}} \\
&= \mathcal{O}\left(\frac{1}{mn}\right)(v(|z_{1}|)+(v(|z_{2}|))\left( \frac{\omega
(h_{1})}{v(h_{1})}+\frac{\omega (h_{2})}{v(h_{2})}\right) \int_{\pi
-h_{1}}^{\pi }\frac{dt_{1}}{\left( t_{1}+h_{1}\right) ^{2}}\int_{h_{2}}^{\pi
}\frac{dt_{2}}{t_{2}^{2}} \\
&= \mathcal{O}\left( \frac{1}{m^{2}}\right) (v(|z_{1}|)+(v(|z_{2}|))\left( 
\frac{\omega (h_{1})}{v(h_{1})}+\frac{\omega (h_{2})}{v(h_{2})}\right) ,
\end{align*}%
as far as for $\Vert J_{441}^{\left( 32\right) }\Vert _{p}$ we obtain
\begin{align*}
\Vert J_{441}^{\left( 32\right) }\Vert _{p} &= \mathcal{O}\left(\frac{1}{mn}%
\right)\int_{\pi -h_{1}}^{\pi }\!\!\int_{h_{2}}^{\pi }\left\Vert
H_{x,z_{1},y,z_{2}}(t_{1}+h_{1},t_{2}+h_{2})\right\Vert _{p} \\
& \quad \times \frac{1}{\left( t_{1}+h_{1}\right) ^{2}}\bigg(\frac{1}{t_{2}^{2}}-%
\frac{1}{\left( t_{2}+h_{2}\right) ^{2}}\bigg)dt_{1}dt_{2} \\
&= \mathcal{O}\left(\frac{1}{mn}\right)\int_{\pi -h_{1}}^{\pi }\!\!\int_{h_{2}}^{\pi
}(v(|z_{1}|)+(v(|z_{2}|)) \\
& \quad \times \left( \frac{\omega (t_{1}+h_{1})}{v(t_{1}+h_{1})}+\frac{\omega
(t_{2}+h_{2})}{v(t_{2}+h_{2})}\right) \frac{1}{\left( t_{1}+h_{1}\right) ^{2}%
}\frac{h_{2}(2t_{2}+h_{2})}{t_{2}^{2}\left( t_{2}+h_{2}\right) ^{2}}%
dt_{1}dt_{2} \\
&= \mathcal{O}\left(\frac{1}{mn^{2}}\right)(v(|z_{1}|)+(v(|z_{2}|))\bigg(\frac{\omega
(\pi +h_{1})}{v(\pi +h_{1})}\int_{h_{2}}^{\pi }\frac{dt_{2}}{t_{2}^{3}} 
+\frac{\omega (2h_{2})}{2h_{2}v(h_{2})}\int_{h_{2}}^{\pi }\frac{dt_{2}}{%
t_{2}^{2}}\bigg)h_{1} \\
&= \mathcal{O}\left( \frac{1}{m^{2}}\right) (v(|z_{1}|)+(v(|z_{2}|))\left( 
\frac{\omega (\pi )}{v(\pi )}+\frac{\omega (h_{2})}{v(h_{2})}\right) \\
&= \mathcal{O}\left( \frac{1}{m}\right) (v(|z_{1}|)+(v(|z_{2}|))\left( \frac{%
\omega (h_{1})}{v(h_{1})}+\frac{\omega (h_{2})}{v(h_{2})}\right) .
\end{align*}

Hence, we obtain
\begin{equation*}
\Vert J_{441}^{\left( 3\right) }\Vert _{p}=\mathcal{O}\left( \frac{1}{m}%
\right) (v(|z_{1}|)+(v(|z_{2}|))\left( \frac{\omega (h_{1})}{v(h_{1})}+\frac{%
\omega (h_{2})}{v(h_{2})}\right) .
\end{equation*}

Consequently, we get 
\begin{equation}
\Vert J_{44}^{\left( 3\right) }\Vert _{p}=\mathcal{O}\left( \frac{1}{m}%
\right) (v(|z_{1}|)+(v(|z_{2}|))\left( \frac{\omega (h_{1})}{v(h_{1})}+\frac{%
\omega (h_{2})}{v(h_{2})}\right) .  \label{eqb4}
\end{equation}

By the same way, we obtain 
\begin{equation}
\Vert J_{44}^{\left( 5\right) }\Vert _{p}=\mathcal{O}\left( \frac{1}{n}%
\right) (v(|z_{1}|)+(v(|z_{2}|))\left( \frac{\omega (h_{1})}{v(h_{1})}+\frac{%
\omega (h_{2})}{v(h_{2})}\right) .  \label{eqb5}
\end{equation}

For $J_{44}^{\left( 6\right) }$, we have 
\begin{align}\label{eqb6}
\begin{split}
\Vert J_{44}^{\left( 6\right) }\Vert _{p}& = \mathcal{O}\left( \frac{1}{mn}%
\right) \int_{0}^{h_{1}}\!\!\int_{0}^{h_{2}}\left\Vert
H_{x,z_{1},y,z_{2}}(t_{1}+h_{1},t_{2}+h_{2})\right\Vert _{p}\frac{%
(mnt_{1}t_{2})^{2}dt_{1}dt_{2}}{\left( (t_{1}+h_{1})(t_{2}+h_{2})\right) ^{2}%
}\\
&= \mathcal{O}\left( mn\right)
\int_{0}^{h_{1}}\!\!\int_{0}^{h_{2}}(v(|z_{1}|)+(v(|z_{2}|))\left( \frac{%
\omega (t_{1}+h_{1})}{v(t_{1}+h_{1})}+\frac{\omega (t_{2}+h_{2})}{%
v(t_{2}+h_{2})}\right) dt_{1}dt_{2} \\
&= \mathcal{O}\left( 1\right) (v(|z_{1}|)+(v(|z_{2}|))\left( \frac{\omega
(h_{1})}{v(h_{1})}+\frac{\omega (h_{2})}{v(h_{2})}\right) . 
\end{split}
\end{align}
Similarly as in the estimation of $J_{44}^{\left( 2\right) }$, we have%
\begin{align*}
J_{44}^{\left( 7\right) } &= \frac{1}{2mn\pi ^{2}}\int_{0}^{h_{1}}%
\int_{h_{2}}^{\pi }H_{x,z_{1},y,z_{2}}(t_{1}+h_{1},t_{2}+h_{2})\frac{%
S(t_{1},t_{2})}{\left( (t_{1}+h_{1})(t_{2}+h_{2})\right) ^{2}}dt_{1}dt_{2} \\
& \quad -\frac{1}{2mn\pi ^{2}}\int_{0}^{h_{1}}\int_{0}^{\pi
-h_{2}}H_{x,z_{1},y,z_{2}}(t_{1}+h_{1},t_{2}+2h_{2})\frac{S(t_{1},t_{2})}{%
\left( (t_{1}+h_{1})(t_{2}+2h_{2})\right) ^{2}}dt_{1}dt_{2} \\
&= \frac{1}{2mn\pi ^{2}}\bigg(\int_{0}^{h_{1}}\int_{h_{2}}^{\pi }\bigg(\frac{%
H_{x,z_{1},y,z_{2}}(t_{1}+h_{1},t_{2}+h_{2})}{\left(
(t_{1}+h_{1})(t_{2}+h_{2})\right) ^{2}}-\frac{%
H_{x,z_{1},y,z_{2}}(t_{1}+h_{1},t_{2}+2h_{2})}{\left(
(t_{1}+h_{1})(t_{2}+2h_{2})\right) ^{2}}\bigg) \\
&\qquad \times
S(t_{1},t_{2})dt_{1}dt_{2} \\
& \quad -\int_{0}^{h_{1}}%
\int_{0}^{h_{2}}H_{x,z_{1},y,z_{2}}(t_{1}+h_{1},t_{2}+2h_{2})\frac{%
S(t_{1},t_{2})}{\left( (t_{1}+h_{1})(t_{2}+2h_{2})\right) ^{2}}dt_{1}dt_{2}
\\
& \quad +\int_{0}^{h_{1}}\int_{\pi -h_{2}}^{\pi
}H_{x,z_{1},y,z_{2}}(t_{1}+h_{1},t_{2}+2h_{2})\frac{S(t_{1},t_{2})}{\left(
(t_{1}+h_{1})(t_{2}+2h_{2})\right) ^{2}}dt_{1}dt_{2}\bigg)
\end{align*}%
and consequently%
\begin{equation}
\Vert J_{44}^{\left( 7\right) }\Vert _{p}=\mathcal{O}\left( 1\right)
(v(|z_{1}|)+(v(|z_{2}|))\left( \frac{\omega (h_{1})}{v(h_{1})}+\frac{\omega
(h_{2})}{v(h_{2})}\right) .  \label{eqb7}
\end{equation}

Analogously, 
\begin{equation}
\Vert J_{44}^{\left( 9\right) }\Vert _{p}=\mathcal{O}\left( 1\right)
(v(|z_{1}|)+(v(|z_{2}|))\left( \frac{\omega (h_{1})}{v(h_{1})}+\frac{\omega
(h_{2})}{v(h_{2})}\right) .  \label{eqb8}
\end{equation}

For $J_{44}^{\left( 8\right) }$, we have 
\begin{align}\label{eqb9}
\begin{split}
\Vert J_{44}^{\left( 8\right) }\Vert _{p} &= \mathcal{O}\left( \frac{1}{mn}%
\right) \int_{0}^{h_{1}}\!\!\int_{\pi -h_{2}}^{\pi }\left\Vert
H_{x,z_{1},y,z_{2}}(t_{1}+h_{1},t_{2}+h_{2})\right\Vert _{p}\frac{%
(mt_{1})^{2}dt_{1}dt_{2}}{\left( (t_{1}+h_{1})(t_{2}+h_{2})\right) ^{2}} \\
&= \mathcal{O}\left( \frac{m}{n}\right) \int_{0}^{h_{1}}\!\!\int_{\pi
-h_{2}}^{\pi }(v(|z_{1}|)+(v(|z_{2}|)) \\
& \quad \times \left( \frac{\omega (t_{1}+h_{1})}{v(t_{1}+h_{1})}+\frac{\omega
(t_{2}+h_{2})}{v(t_{2}+h_{2})}\right) \frac{dt_{1}dt_{2}}{\left(
t_{2}+h_{2}\right) ^{2}} \\
&= \mathcal{O}\left( \frac{m}{n}\right) (v(|z_{1}|)+(v(|z_{2}|)) 
 \left( \frac{\omega (h_{1})}{v(h_{1})}h_{2}+\int_{\pi -h_{2}}^{\pi }%
\frac{\omega (t_{2}+h_{2})}{v(t_{2}+h_{2})}\frac{dt_{2}}{\left(
t_{2}+h_{2}\right) ^{2}}\right) h_{1} \\
&= \mathcal{O}\left( \frac{1}{n}\right) (v(|z_{1}|)+(v(|z_{2}|)) 
 \left( \frac{\omega (h_{1})}{v(h_{1})}h_{2}+\int_{\pi }^{\pi +h_{2}}%
\frac{\omega (\theta _{2})}{v(\theta _{2})}\frac{d\theta _{2}}{\theta
_{2}^{2}}\right)\\
&= \mathcal{O}\left( \frac{1}{n^{2}}\right) (v(|z_{1}|)+(v(|z_{2}|))\left( 
\frac{\omega (h_{1})}{v(h_{1})}+\frac{\omega (\pi )}{v(\pi )}\right) \\
&= \mathcal{O}\left( \frac{1}{n}\right) (v(|z_{1}|)+(v(|z_{2}|))\left( \frac{%
\omega (h_{1})}{v(h_{1})}+\frac{\omega (h_{2})}{v(h_{2})}\right) . 
\end{split}
\end{align}

By analogy, we also get 
\begin{equation}
\Vert J_{44}^{\left( 11\right) }\Vert _{p}=\mathcal{O}\left( \frac{1}{m}%
\right) (v(|z_{1}|)+(v(|z_{2}|))\left( \frac{\omega (h_{1})}{v(h_{1})}+\frac{%
\omega (h_{2})}{v(h_{2})}\right) .  \label{eqb10}
\end{equation}

Now we need to give the upper bound for $\Vert J_{44}^{\left( 10\right)
}\Vert _{p}$. We have 
\begin{align}\label{eqb11}
\begin{split}
\Vert J_{44}^{\left( 10\right) }\Vert _{p} &= \mathcal{O}\left( \frac{1}{mn}%
\right) \int_{h_{1}}^{\pi }\!\!\int_{\pi -h_{2}}^{\pi }\left\Vert
H_{x,z_{1},y,z_{2}}(t_{1}+h_{1},t_{2}+h_{2})\right\Vert _{p}\frac{%
dt_{1}dt_{2}}{\left( (t_{1}+h_{1})(t_{2}+h_{2})\right) ^{2}} \\
&= \mathcal{O}\left( \frac{1}{mn}\right) \int_{h_{1}}^{\pi }\!\!\int_{\pi
-h_{2}}^{\pi }(v(|z_{1}|)+(v(|z_{2}|)) \\
& \quad \times \left( \frac{\omega (t_{1}+h_{1})}{v(t_{1}+h_{1})}+\frac{\omega
(t_{2}+h_{2})}{v(t_{2}+h_{2})}\right) \frac{dt_{1}dt_{2}}{\left(
t_{1}+h_{1}\right) ^{2}}\\
&= \mathcal{O}\left( \frac{1}{mn}\right) (v(|z_{1}|)+(v(|z_{2}|))\left( 
\frac{\omega (\pi )}{v(\pi )}\frac{m}{n}+m\int_{\pi }^{\pi +h_{2}}\frac{%
\omega (u_{2})du_{2}}{v(u_{2})u_{2}^{2}}\right) \\
&= \mathcal{O}\left( \frac{1}{n^{2}}\right) (v(|z_{1}|)+(v(|z_{2}|))\frac{%
\omega (\pi )}{v(\pi )} \\
&= \mathcal{O}\left( \frac{1}{n}\right) (v(|z_{1}|)+(v(|z_{2}|))\frac{\omega
(h_{2})}{v(h_{2})}.
\end{split}
\end{align}

By analogy, we obtain 
\begin{equation}
\Vert J_{44}^{\left( 12\right) }\Vert _{p}=\mathcal{O}\left( \frac{1}{m}%
\right) (v(|z_{1}|)+(v(|z_{2}|))\frac{\omega (h_{1})}{v(h_{1})}.
\label{eqb12}
\end{equation}

Finally, we have%
\begin{align}\label{eqb13}
\begin{split}
\Vert J_{44}^{\left( 13\right) }\Vert _{p} &= \mathcal{O}\left( \frac{1}{mn}%
\right) \int_{\pi -h_{1}}^{\pi }\!\!\int_{\pi -h_{2}}^{\pi }\left\Vert
H_{x,z_{1},y,z_{2}}(t_{1}+h_{1},t_{2}+h_{2})\right\Vert _{p}\frac{%
dt_{1}dt_{2}}{\left( (t_{1}+h_{1})(t_{2}+h_{2})\right) ^{2}} \\
&= \mathcal{O}\left( \frac{1}{mn}\right) \int_{\pi -h_{1}}^{\pi
}\!\!\int_{\pi -h_{2}}^{\pi }(v(|z_{1}|)+(v(|z_{2}|)) \\
& \quad \times \left( \frac{\omega (t_{1}+h_{1})}{v(t_{1}+h_{1})}+\frac{\omega
(t_{2}+h_{2})}{v(t_{2}+h_{2})}\right) \frac{dt_{1}dt_{2}}{\left(
(t_{1}+h_{1})(t_{2}+h_{2})\right) ^{2}} \\
&= \mathcal{O}\left( \frac{1}{mn}\right) (v(|z_{1}|)+(v(|z_{2}|)) \\
& \quad \times \left( h_{2}\int_{\pi -h_{1}}^{\pi }\frac{\omega (t_{1}+h_{1})}{%
v(t_{1}+h_{1})}\frac{dt_{1}}{(t_{1}+h_{1})^{2}}+h_{1}\int_{\pi -h_{2}}^{\pi }%
\frac{\omega (t_{2}+h_{2})}{v(t_{2}+h_{2})}\frac{dt_{2}}{(t_{2}+h_{2})^{2}}%
\right) \\
&= \mathcal{O}\left( \frac{1}{mn}\right) (v(|z_{1}|)+(v(|z_{2}|))\left(
h_{2}\int_{\pi }^{\pi +h_{1}}\frac{\omega (\theta _{1})}{v(\theta _{1})}%
\frac{d\theta _{1}}{\theta _{1}^{2}}+h_{1}\int_{\pi }^{\pi +h_{2}}\frac{%
\omega (\theta _{2})}{v(\theta _{2})}\frac{d\theta _{2}}{\theta _{2}^{2}}%
\right) \\
&= \mathcal{O}\left( \frac{1}{\left( mn\right) ^{2}}\right)
(v(|z_{1}|)+(v(|z_{2}|))\frac{\omega (\pi )}{v(\pi )}\\
& = \mathcal{O}\left( \frac{1}{mn}\right) (v(|z_{1}|)+(v(|z_{2}|))\left( \frac{%
\omega (h_{1})}{v(h_{1})}+\frac{\omega (h_{2})}{v(h_{2})}\right) .
\end{split}
\end{align}

Whence, using \eqref{eqb1}-\eqref{eqb13}, we obtain 
\begin{equation*}
\left\Vert J_{44}\right\Vert _{p}=\mathcal{O}\left( 1\right)
(v(|z_{1}|)+(v(|z_{2}|))\left( \frac{\omega (h_{1})}{v(h_{1})}+\frac{\omega
(h_{2})}{v(h_{2})}\right) .
\end{equation*}%
Combining partial estimates, we get%
\begin{equation*}
\sup_{z_{1}\neq 0,\,\,z_{2}\neq 0}\frac{\Vert
D_{m,n}(x+z_{1},y+z_{2})-D_{m,n}(x,y)\Vert _{p}}{v(|z_{1}|)+v(|z_{2}|)}=%
\mathcal{O}\left( 1\right) \left( \frac{\omega (h_{1})}{v(h_{1})}+\frac{%
\omega (h_{2})}{v(h_{2})}\right) .
\end{equation*}%
Procceding as in the similar lines, it can be proved that%
\begin{equation*}
\Vert D_{m,n}\Vert _{p}=\mathcal{O}\left( 1\right) \left( \omega
(h_{1})+\omega (h_{2})\right) .
\end{equation*}%
Hence%
\begin{equation*}
\Vert D_{m,n}\Vert _{p}^{(v,v)}=\mathcal{O}\left( 1\right) \left( \frac{%
\omega (h_{1})}{v(h_{1})}+\frac{\omega (h_{2})}{v(h_{2})}\right)
\end{equation*}%
and this ends our proof.
\end{proof}}

Now we can finish this section by deriving some particular results. For this, we need to  
specialize the functions $\omega(t)$ and $v(t)$ in our theorem. In fact, taking $\omega(t)=t^{\alpha}$ and $v(t)=t^{\beta}$, $0\leq \beta< \alpha\leq 1$, in Theorem \ref{the01} we get:

\begin{corollary} 
If $f\in \text{Lip}(\alpha,\beta, p)$, $p\geq 1$, $0\leq \beta< \alpha\leq 1$,  then 
\begin{equation*}
\Vert \sigma _{m,2m;n,2n}(f)-f\Vert _{p}^{(v,v)}= \mathcal{O} \left( \frac{1
}{m^{\alpha -\beta}}+\frac{1
}{n^{\alpha -\beta}}\right) 
\end{equation*}%
for all $m,n\in 
\mathbb{N}
$.
\end{corollary} 

For $p=\infty$ in above corollary, we obtain the following.

\begin{corollary} 
If $f\in \text{H}_{(\alpha,\beta)}$, $0\leq \beta< \alpha\leq 1$,  then 
\begin{equation*}
\Vert \sigma _{m,2m;n,2n}(f)-f\Vert _{\alpha,\beta}= \mathcal{O} \left( \frac{1
}{m^{\alpha -\beta}}+\frac{1
}{n^{\alpha -\beta}}\right) 
\end{equation*}%
for all $m,n\in 
\mathbb{N}
$.
\end{corollary}

\section{A generalization of Theorem \ref{the01}}

We denote by $W(L^p((-\pi,\pi)^2); \beta_1, \beta_2)$ the weighted space $L^p((-\pi,\pi)^2)$ with weight function $\left|\sin \left(\frac{x}{2}\right)\right|^{\beta_1 p}\left|\sin \left(\frac{y}{2}\right)\right|^{\beta_2 p}$, ($\beta_1,\beta_2\geq 0$), and endowed with norm 
$$\|f\|_{p;\beta_1,\beta_2}:=\left(\frac{1}{(2\pi)^2}\int_{-\pi}^{\pi}\int_{-\pi}^{\pi}\left|f(x,y)\right|^p\left|\sin \left(\frac{x}{2}\right)\right|^{\beta_1 p}\left|\sin \left(\frac{y}{2}\right)\right|^{\beta_2 p}dxdy\right)^p$$
for $1\leq p<\infty$, and
$$\|f\|_{p;\beta_1,\beta_2}:=\text{ess}\!\!\!\!\!\sup_{-\pi\leq x,y \leq \pi }\left\{\left|f(x,y)\right|\left|\sin \left(\frac{x}{2}\right)\right|^{\beta_1 }\left|\sin \left(\frac{y}{2}\right)\right|^{\beta_2 }\right\}$$
for $p=\infty$ (for $\beta_1=0,\beta_2=0$ see \cite{US}).

Acting accordingly, we define the space $H_{p;\beta_1,\beta_2}^{(\omega_1, \omega_2)}$ by 
\begin{equation*}
H_{p;\beta_1,\beta_2}^{(\omega_1, \omega_2 )}:=\left\{f\in W(L^p((-\pi ,\pi )^2); \beta_1, \beta_2), p\geq 1:
A(f;\omega_1 , \omega_2 ;\beta_1 ,\beta_2 )<\infty \right\},
\end{equation*}
where 
\begin{equation*}
A(f; \omega_1 , \omega_2 ;\beta_1 ,\beta_2 ):=\sup_{t_1\neq 0,\,\,t_2\neq 0 }\frac{\|f(x +t_1,y
+t_2)-f(x,y)\|_{p ;\beta_1 ,\beta_2 }}{\omega_1 (|t_1|)+\omega_2 (|t_2|)}
\end{equation*}
and the norm in the space $H_{p ;\beta_1 ,\beta_2 }^{(\omega_1, \omega_2)}$ is defined by 
\begin{equation*}
\|f\|_{p ;\beta_1 ,\beta_2 }^{(\omega_1, \omega_2)}:=\|f\|_{p ;\beta_1 ,\beta_2 }+A(f;\omega_1, \omega_2  ;\beta_1 ,\beta_2 ).
\end{equation*}

We take $k=rm$ and $\ell =qn$ in (\ref{eq1}), to obtain 
\begin{align*}
\sigma_{m,rm;n,qn}:=&\left( 1+\frac{1}{r}\right)\left( 1+\frac{1}{q}%
\right)\sigma_{m(r+1)-1,n(q+1)-1}-\left( 1+\frac{1}{r}\right)\frac{1}{q}%
\sigma_{m(r+1)-1,n-1} \nonumber \\
& -\frac{1}{r}\left( 1+\frac{1}{q}\right)\sigma_{m-1,n(q+1)-1}+\frac{1}{%
rq}\sigma_{m-1,n-1},
\end{align*}
where $r,q\in \{2,4,6,\dots\}$.

Hence, we name the mean $\sigma _{m,rm;n,qn}(f;x,y)$ as {\it Double Even-Type Delayed Arithmetic Mean} for $%
S_{k,\ell }(f;x,y)$, which can be represented in its integral form

\begin{equation*}
\sigma_{m,rm;n,qn}(f;x,y)=\frac{1}{mn\pi^2}\int_{0}^{\pi}\int_{0}^{%
\pi}h_{x,y}(t_1,t_2)L_{m,rm;n,qn}(t_1,t_2)dt_1dt_2,
\end{equation*}
where
$$L_{m,rm;n,qn}(t_1,t_2):=\frac{4}{rq}\frac{\sin \frac{\left(r+2\right)mt_1}{2}\sin \frac{rmt_1}{2}\sin \frac{\left(q+2\right)nt_2}{2}\sin \frac{qnt_2}{2}}{\left( 4\sin \frac{t_1}{2}\sin \frac{t_2%
}{2}\right)^2}.$$

Further, we establish a more general theorem on the degree of approximation of function $f$ belonging to $H_{p ;\beta_1 ,\beta_2 }^{(\omega, \omega)}$ with norm $\|f\|_{p ;\beta_1 ,\beta_2 }^{(v,v)}$ by Double Even-Type Delayed Arithmetic Mean $\sigma_{m,rm;n,qn}(f;x,y)$. It represent a twofold generalization of Theorem \ref{the01}, both in context of the considered space and the mean.

\begin{theorem}
\label{the05} Let $\omega $ and $v$ be moduli of continuity so that $\frac{%
\omega (t)}{v(t)}$ is non-decreasing in $t$. If $f\in H_{p ;\beta_1 ,\beta_2 }^{(\omega, \omega)}$, $p\geq 1$, then 
\begin{equation*}
\Vert \sigma_{m,rm;n,qn}(f)-f\Vert _{p ;\beta_1 ,\beta_2 }^{(v,v)}= \mathcal{O} \left( r\frac{\omega (h_{1})%
}{v(h_{1})}+q\frac{\omega (h_{2})}{v(h_{2})}\right) ,
\end{equation*}%
where $h_{1}=\frac{\pi }{m}$, $h_{2}=\frac{\pi }{n}$, $m,n\in 
\mathbb{N}
$, and $r,q\in \{2,4,6,\dots\}$.
\end{theorem}

\begin{proof}
The proof can be done in the same lines as the proof of Theorem \ref{the01}. We omit the details. 
\end{proof}

\begin{remark}
For $\beta_1 =0$, $\beta_2=0$, $r=2$ and $q=2$, Theorem \ref{the05} reduces exactly to Theorem \ref{the01}.
\end{remark}

\section{Conclusions}

For one dimension, a degree of approximation of a function in the space $H_{p}^{(\omega)}$ has been given by several authors, see \cite{DGR,DNR,D,UD2,KIM1,XhK,L,NDR,SS} and some other references already mentioned here. Inspired by these papers, we have given two corresponding results using the second (even) type double delayed arithmetic means of the Fourier series of a function from the space $H_{p}^{(\omega , \omega)}$ ($H_{p ;\beta_1 ,\beta_2 }^{(\omega, \omega)}$).

\label{lastpage}

\end{document}